\documentclass[10pt,journal,twoside]{IEEEtran}

\usepackage{amsmath,amssymb,amsthm,mathtools,bm,mathrsfs}
\usepackage{graphicx}
\usepackage{xcolor}
\usepackage{booktabs}
\usepackage{enumitem}
\usepackage{cite}
\usepackage[hidelinks]{hyperref}
\usepackage{microtype}
\allowdisplaybreaks
\usepackage[caption=false,font=footnotesize]{subfig}

\newtheorem{theorem}{Theorem}
\newtheorem{proposition}{Proposition}
\newtheorem{lemma}{Lemma}
\newtheorem{corollary}{Corollary}
\theoremstyle{definition}
\newtheorem{assumption}{Assumption}

\theoremstyle{remark}
\newtheorem{remark}{Remark}

\newcommand{\R}{\mathbb{R}}

\newcommand{\diag}{\operatorname{diag}}

\newcommand{\dd}{\mathrm{d}}

\title{Mean--Covariance Turnpikes in Wasserstein Distributionally Robust Linear-Quadratic Control}
\author{Yanzhi~Wu, Zhengping~Ji%
\thanks{This work is supported partly by ERC-2022-ADG-101096251-CoDeFeL grant and by the Alexander von Humboldt Foundation through a Humboldt Research Fellowship. \textit{(Corresponding author: Zhengping Ji.)}}
\thanks{Yanzhi Wu is with the School of Electrical Engineering, Southwest Jiaotong University, Chengdu, China, and with the Chair for Dynamics, Control, Machine Learning \& Numerics (Alexander von Humboldt Professorship), Department of Mathematics, Friedrich-Alexander-Universit\"at Erlangen-N\"urnberg, 91058 Erlangen, Germany (email: wyzcontrolmath@139.com).}%
\thanks{Zhengping Ji is with the Chair for Dynamics, Control, Machine Learning \& Numerics (Alexander von Humboldt Professorship), Department of Mathematics, Friedrich-Alexander-Universit\"at Erlangen-N\"urnberg, 91058 Erlangen, Germany (e-mail: zhengping.ji@fau.de).}}
\begin{document}
\maketitle

\begin{abstract}
We study long-horizon Wasserstein-penalized minimax control for discrete-time stochastic linear systems with empirical disturbance data, in which adversarial disturbance distributions induce time-varying mean and covariance dynamics, making standard turnpike arguments not directly applicable. For possibly uncentered data, we characterize the generally nonzero mean reference through a reduced convex-concave Hamiltonian saddle problem. We prove horizon-uniform, two-sided exponential turnpike estimates for the mean state, adjoint, control, worst-case disturbance mean, and closed-loop covariance, showing that they spend the majority of time near static references when the horizon is long. We further construct a hybrid policy combining time-independent affine feedback with finite-horizon steering over a terminal layer, proving that its worst-case cost gap decays exponentially with the terminal-layer length uniformly in the horizon, which helps reducing the computation cost for long-horizon robust controls. Numerical examples illustrate the estimates and their
dependence on the Wasserstein penalty.
\end{abstract}

\section{Introduction}

\subsection{Background}

Stochastic optimal control traditionally relies on prescribed probabilistic information about system disturbances. In classical linear-quadratic control, this takes the form of a nominal distribution used to minimize an expected cost. However, this assumption is restrictive in modern data-driven settings. In practice, disturbance distributions are estimated from finite empirical samples, meaning their quality is fundamentally limited by data volume, sensor fidelity, and operating conditions. If the assumed nominal distribution deviates from reality, the resulting controller may suffer severe performance degradation and yield unsafe or undesirable closed-loop behavior
\cite{Yang2021TAC,HakobyanYang2024}.

Distributionally robust control provides a powerful minimax approach to overcome this difficulty.
Instead of optimizing with respect to a single nominal distribution, the
controller minimizes the cost against a worst-case disturbance distribution in an
ambiguity set.  This viewpoint bridges stochastic and robust control by using
available statistical information while hedging against errors in the empirical distribution used for controller design.  Ambiguity sets have been constructed using relative entropy, moment constraints, total variation distance, and the Wasserstein metric
\cite{PetersenJamesDupuis2000,UgrinovskiiPetersen2002,
DelageYe2010,WiesemannKuhnSim2014,VanParysKuhnGoulartMorari2016,
CoulsonLygerosDorfler2022,Yang2021TAC}.  Among
these choices, the Wasserstein metric is especially suitable when the nominal
distribution is empirical.  
Because it quantifies the optimal transport cost between distributions, a Wasserstein ambiguity set can be centered directly on discrete empirical data, enabling highly effective data-driven linear-quadratic solutions
\cite{MohajerinEsfahaniKuhn2018,GaoKleywegt2023,KuhnShafieeWiesemann2025,
Yang2021TAC,KimYang2023,HakobyanYang2024}.

While Wasserstein-penalized robust control provides strong performance guarantees, solving these minimax linear-quadratic problems over long finite horizons introduces significant computational burdens and raises complementary questions about the structure of optimal trajectories. In classical optimal control, the turnpike property describes the phenomenon that long-horizon optimal trajectories spend the vast majority of time near a time-independent steady-state solution, deviating only during brief initial and terminal boundary layers
\cite{Grune2013,TrelatZuazua2015,TrelatZhangZuazua2018}. To establish this property for Wasserstein-penalized robust
control is of both theoretical and practical importance, as it will explain the long-horizon behavior of optimal trajectories and controls, and solve the computational bottleneck of long horizons by reducing the dynamic mean variables and covariance evolution to static optimizations, and providing a basis for approximating these variables by their time-independent references over the interior of a long horizon.

\subsection{Related work}

The literature on Wasserstein distributionally robust control has mainly addressed policy construction, tractable reformulations, performance guarantees, and stability. For linear systems, \cite{Yang2021TAC} developed a dynamic-programming framework for Wasserstein-penalized linear-quadratic problems, and \cite{KimYang2020,KimYang2023} derived Riccati recursions and optimal policies for such problems. \cite{HakobyanYang2024} extended this to partially observable systems. Recent studies have considered constrained and infinite-horizon control with Wasserstein ambiguity sets
\cite{BrouillonMartinLygerosDorflerFerrariTrecate2025}, optimal-transport-based uncertainty propagation
\cite{AolariteiLanzettiChenDorfler2025}, and distributionally robust regret-optimal control \cite{KarginHajarMalikHassibi2024}.
Distributionally robust linear-quadratic and linear-quadratic-Gaussian formulations have also been studied in \cite{TaskesenIancuKocyigitKuhn2023,
SchobiLanzettiDorflerDAndreaTerpin2026}. Related work on distributionally robust model predictive control has addressed output feedback, closed-loop guarantees, and scalable algorithms
\cite{LiGuanDaiDuan2024,McAllisterEsfahani2025}.
However, to the best of our knowledge, the existing literature does not answer whether the long finite-horizon robust closed loop exhibits a horizon-uniform turnpike, especially whether the adversarially induced second-order dynamics have a covariance turnpike, and did not study the dependence of solutions on horizon length, which, as we will show in this paper, can help reducing the computational complexity of distributionally robust control.

In the analysis of long-horizon behavior of optimal control problems, turnpike theory has mainly been developed for deterministic or standard linear
stochastic optimal control problems, using Hamiltonian dynamics, dissipativity or matrix inequalities to derive horizon-independent estimates
\cite{DammGruneStielerWorthmann2014,
GruneGuglielmi2018,BerberichEtAl2018,SchiesslBaumannFaulwasserGrune2025}.
Stochastic turnpike properties have also been extended to probability measures, mean-field couplings, zero-sum differential games, and partially observed systems \cite{SunWangYong2022,SunYong2024,SunYong2025Game,HuWu2026,Schie2025}. 
However, extending these ideas to the present distributionally robust setting faces two structural difficulties. First, unlike the standard scheme where the disturbance distribution is prescribed, the stagewise selection of a worst-case disturbance distribution in the data-driven setting generates a time-varying closed-loop covariance recursion whose coefficients depend on the finite-horizon Riccati sequence. This dynamic complexity is not captured by deterministic turnpike theory or by an analysis of the mean dynamics alone. Second, in practical data-driven applications, the empirical disturbance mean need not vanish. Since artificially recentering the samples changes the empirical reference distribution entering the Wasserstein penalty, the origin cannot in general serve as the stationary mean reference. The appropriate nonzero reference for the state, adjoint, control, and worst-case disturbance mean must therefore be characterized consistently with the original empirical data.

\subsection{Contributions}


Our main contributions are summarized as follows.

\begin{itemize}
	
	\item {\it Static mean reference for uncentered data.}
    We show that a nonzero empirical disturbance mean generates affine terms in
	the value functions and feedforward terms in the optimal policies, allowing us to characterize the static mean reference as the unique saddle point of a reduced convex--concave Hamiltonian problem. This formulation avoids artificial recentering of the empirical distribution and identifies the stationary reference around which the finite-horizon mean variables exhibit their turnpike behavior.

  
    
	
	\item {\it Two-sided mean turnpike estimates.} 
    By establishing a uniform product estimate for the time-varying closed-loop matrices, we derive two-sided exponential estimates, uniform with respect to the horizon, for the mean state, mean adjoint, mean control input, and mean of the worst-case disturbance distribution. Thus, outside horizon-independent initial and terminal boundary layers, the finite-horizon mean variables remain exponentially close to the static mean reference.
    
    
	
	\item {\it Covariance turnpike induced by the worst-case distribution policy}.
     We derive an exact closed-loop covariance recursion whose coefficients depend on the finite-horizon Riccati sequence and analyze it relative to a limiting Lyapunov equation. This yields a two-sided exponential covariance turnpike estimate, uniform with respect to the horizon. Together with the mean estimates, it yields corresponding bounds for expected quadratic functions and a 2-Wasserstein estimate for Gaussian moment surrogates.

     \item {\it Hybrid quasi-turnpike policy with a performance guarantee}. Motivated by the turnpike estimates, we construct a hybrid control policy that uses the time-independent affine feedback over the first $N-q$ stages and retains the finite-horizon optimal coefficients over the final $q$ stages. We prove that, for sufficiently large $q$, its worst-case finite-horizon cost gap decays exponentially with $q$, independent of the horizon $N$. Hence, a fixed terminal-layer length suffices to achieve a prescribed cost accuracy uniformly as the horizon grows.
     
	
\end{itemize}

\subsection{Organization of the paper}

The rest of the paper is organized as follows. Section \ref{sec:problem} formulates the finite-horizon minimax problem and describes our main result.
Section \ref{sec:bellman-riccati} derives the affine-quadratic expression of the value functions and
studies the associated Riccati equation. Section \ref{sec:static} introduces the static mean reference and the finite-horizon mean-state and
adjoint equations. Section \ref{sec:turnpike} establishes the mean and covariance turnpike estimates, their consequences for quadratic functions and Gaussian moment surrogates, and a hybrid quasi-turnpike policy with an exponential finite-horizon cost estimate, which are illustrated by numerical experiments in Section \ref{sec:numerical-experiments}.
Section \ref{sec:conclusion} concludes the paper.

\section{Problem formulation and main result}
\label{sec:problem}

\subsection{System model}
\label{subsec:finite-horizon-problem}

Consider the discrete-time stochastic linear system
\begin{equation}
	x_{t+1}=Ax_t+Bu_t+\Xi w_t,
	\qquad t\ge 0,
	\label{eq:sys}
\end{equation}
where \(x_t\in\mathbb R^n\) is the state, \(u_t\in\mathbb R^m\) is the control input, and \(w_t\in\mathbb R^k\) is a random disturbance vector.  The system matrices are
\(
A\in\mathbb R^{n\times n},
B\in\mathbb R^{n\times m},
\Xi\in\mathbb R^{n\times k}.
\)

The probability distribution of \(w_t\) is unknown.  Instead, at each stage \(t\), a finite disturbance dataset
\(
\widehat w_t^{(1)},\ldots,\widehat w_t^{(M)}\in\mathbb R^k
\)
is available.  Here \(\widehat w_t^{(i)}\) denotes the \(i\)-th empirical disturbance sample at stage \(t\).  Let \(\delta_a\) denote the Dirac probability measure concentrated at \(a\).  The corresponding empirical distribution is
\begin{equation}
	\nu_t:=\frac1M\sum_{i=1}^M\delta_{\widehat w_t^{(i)}} .
	\label{eq:empirical-law}
\end{equation}
 Its empirical mean and covariance are
\begin{equation}
	\begin{aligned}
		d_t
		&:=
		\frac1M\sum_{i=1}^M\widehat w_t^{(i)},\\
		\Sigma_t
		&:=
		\frac1M\sum_{i=1}^M
		(\widehat w_t^{(i)}-d_t)
		(\widehat w_t^{(i)}-d_t)^\top .
	\end{aligned}
	\label{eq:mean-cov}
\end{equation}

The empirical disturbance samples are retained in their original coordinates, hence,  \(d_t\) is not necessarily zero.

\begin{assumption}
	\label{ass:stationary-empirical}
	The empirical mean and covariance in \eqref{eq:mean-cov} are independent of \(t\), i.e.,
	 $d_t\equiv d,
		\Sigma_t\equiv\Sigma$. The initial state $x_0$ in \eqref{eq:sys} has distribution \(\mu_0\in\mathcal P_2(\mathbb R^n)\), with mean and
	covariance
	\begin{equation}
		\begin{aligned}
			m_0:=\!\!\int \!x\,\dd\mu_0(x),
			~X_0\!:=\!\!\int \!(x-m_0)(x-m_0)^\top\dd\mu_0(x).
		\end{aligned}
		\label{eq:initial-ensemble-moments}
	\end{equation}
	The deterministic initial state is recovered by taking
	\(\mu_0=\delta_{x_0}\).
\end{assumption}

\begingroup
\begin{remark}
\begin{enumerate}
    \item An analogous stationary condition on the nominal disturbance mean and
	covariance is imposed in \cite{HakobyanYang2024}. Assumption~\ref{ass:stationary-empirical}
	requires only the empirical mean and covariance to be independent of the
	stage; the empirical support points $\widehat w_t^{(i)}$ may still vary with the stage.
    \item A random initial state specified by its mean and covariance is also
	considered in \cite{HakobyanYang2024}. The finite-second-moment condition in
	Assumption~\ref{ass:stationary-empirical} is used here to define the initial
	covariance and to support the covariance analysis.
\end{enumerate}
\end{remark}
\endgroup

\subsection{Wasserstein-penalized minimax formulation}
\label{subsec:minimax-formulation}

Next, we define the cost or objective function that we would like to optimize with respect to the dynamics \eqref{eq:sys}. 

Let \(\mathcal P_2(\mathbb R^k)\) denote the set of Borel probability
measures on \(\mathbb R^k\) with finite second moment.  For
\(\mu,\nu\in\mathcal P_2(\mathbb R^k)\), the squared 2-Wasserstein distance measuring their difference is
\begin{equation}
	W_2^2(\mu,\nu)
	:=
	\inf_{\kappa\in\mathcal C(\mu,\nu)}
	\int_{\mathbb R^k\times\mathbb R^k}
	\|w-\widetilde w\|^2
	\,\dd\kappa(w,\widetilde w),
	\label{eq:W2-def}
\end{equation}
where \(\mathcal C(\mu,\nu)\) is the set of couplings whose first and second
marginals are \(\mu\) and \(\nu\), respectively.

Fix time horizon \(N\in\mathbb N\) and consider the finite-horizon stages \(0,\ldots,N-1\). For any index set $I\subseteq\{0,\ldots,N-1\}$, define the state-feedback policy class
\begin{equation*}
	\begin{aligned}
		\Pi_I
		:=
		\{(\pi_r)_{r\in I}:\,
		\pi_r:\mathbb R^n\to\mathbb R^m 
		~\text{is Borel measurable}\},
	\end{aligned}
\end{equation*}
and  the disturbance distribution policy class
\begin{equation*}
	\begin{aligned}
		\Gamma_I:=\Bigl\{
		(\gamma_r)_{r\in I}:&\,
		\gamma_r:\mathbb R^n\times\mathbb R^m
		\to\mathcal P_2(\mathbb R^k),\\
		&\ \gamma_r \text{ is Borel measurable for all } r\in I
		\Bigr\}.
	\end{aligned}
\end{equation*}
 Under
\((\pi,\gamma)\in\Pi_I\times\Gamma_I\), the control input is
\(u_r=\pi_r(x_r)\), the conditional disturbance distribution is
\(\gamma_r(x_r,u_r)\),
and the state evolves according to \eqref{eq:sys}.

For the current state \(x_t=x\) at stage \(t\), define the finite-horizon
Wasserstein-penalized cost with respect to \eqref{eq:sys} as
\begin{equation}
	\begin{aligned}
		J_{x,t}(\pi,\gamma)
		&:=
		\mathbb E^{\pi,\gamma}\Bigg[
		x_N^\top Q_fx_N + \sum_{r=t}^{N-1}
		\Big(
		x_r^\top Qx_r+u_r^\top Ru_r \\
		&
		-\lambda W_2^2(\gamma_r(x_r,u_r),\nu_r)
		\Big) 
		\,\Big|\,x_t=x
		\Bigg],
	\end{aligned}
	\label{eq:truncated-cost}
\end{equation}
where
$Q=Q^\top\succeq0$,
$R=R^\top\succ0$,
$Q_f=Q_f^\top\succeq0$,
and \(\lambda>0\) is the Wasserstein penalty weight. The term \(-\lambda W_2^2(\gamma_r(x_r,u_r),\nu_r)\) penalizes the deviation of the selected disturbance distribution from the empirical distribution $\nu_r$ defined in \eqref{eq:empirical-law}.

Intuitively speaking, to achieve robust control against unknown disturbances, the goal of the controller  $u_r=\pi_r(x_r)$ is to minimize the quadratic cost in \eqref{eq:truncated-cost} over time $N$ in the presence of a worst-case disturbance $\gamma_r(x_r,u_r)$ corresponding to the sampled noisy data.

We define the finite-horizon value function as follows:
\begin{equation}
	V_t(x_t)
	:=
	\inf_{\pi\in\Pi_{t:N-1}}
	\sup_{\gamma\in\Gamma_{t:N-1}}
	J_{x,t}(\pi,\gamma),
	\label{eq:value-definition}
\end{equation}
where
\(\Pi_{t:N-1}:=\Pi_{\{t,\ldots,N-1\}}\) and
\(\Gamma_{t:N-1}:=\Gamma_{\{t,\ldots,N-1\}}\). The terminal value of the value function is
\(
V_N(x_N):=x_N^\top Q_fx_N .
\)
The optimizations in \eqref{eq:value-definition} are over admissible
{control policies $\pi$ and disturbance distribution policies $\gamma$} for which the relevant second moments are finite and
\eqref{eq:truncated-cost} is well defined.

The principal objective of this paper is to characterize the long-horizon mean and covariance behavior of the system \eqref{eq:sys} under the optimal control and disturbance distribution policies associated with the value function \eqref{eq:value-definition}.


\subsection{Main result in brief}

For a horizon \(N\), let \(m_t^*\), \(p_t^*\),
\(\bar{u}_t^*\), \(\bar{w}_t^*\), and \(X_t^*\)
denote, respectively, the optimal mean state and adjoint, the optimal mean
control input, the mean of the worst-case disturbance distribution,
and the optimal closed-loop state covariance. Let \((m^e,p^e,u^e,\bar w^e)\) be the static mean reference defined
by the reduced static Hamiltonian saddle problem in
Proposition~\ref{prop:static-equilibrium}, and let \(X^e\) be the reference covariance defined by
the limiting Lyapunov equation.

Under Assumptions \ref{ass:stationary-empirical},
\ref{ass:psd-W}, and \ref{ass:riccati-penalty}, there exist constants \(C>0\) and
\(\rho\in(0,1)\), independent of \(N\) and \(t\), such that, for
\(t=0,\ldots,N-1\),
\begin{align}
&\|m_t^*-m^e\|
 +\|p_t^*-p^e\|
 +\|\bar{u}_t^*-u^e\|
 \notag\\
&\quad
 +\|\bar{w}_t^*-\bar w^e\|
 +\|X_t^*-X^e\|
 \notag\\
\le&~
 C\rho^t
 \left(
   \|m_0-m^e\|+\|X_0-X^e\|
 \right)
 \notag\\
&\quad+
 C\rho^{N-t-1}
 \left(
   1+\|Q_fm^e-p^e\|
 \right).
\label{eq:main-result-in-brief}
\end{align}
Here \(m_0\) and \(X_0\) are the initial mean and covariance, and
\(Q_f\) is the terminal-state cost matrix. Estimate
\eqref{eq:main-result-in-brief} combines the mean and covariance turnpike
bounds established in Theorems~\ref{thm:affine-mean-turnpike} and
\ref{thm:cov-turnpike}, respectively, and displays the corresponding
initial and terminal boundary terms explicitly.

\section{Value function and Riccati equation}
\label{sec:bellman-riccati}

This section analyzes the value function of the Wasserstein-penalized minimax problem \eqref{eq:sys}\eqref{eq:value-definition} and derives the properties of the Riccati equation, which characterizes the optimizers of the minimax problem and will be essential for turnpike estimates.
For the rest of the paper, the horizon \(N\) is fixed. To avoid overloading the notation, we suppress the dependence of the finite-horizon coefficients and policies on \(N\).

\subsection{Optimal solutions from quadratic value function}
\label{subsec:one-step-bellman}

By Bellman's dynamic programming principle, the value function defined in \eqref{eq:value-definition} satisfies the following recursive equations:
\begin{equation}
	\begin{aligned}
		&V_N(x_N)
		=x_N^\top Q_fx_N,\\
		&V_t(x_t)
		=
		x_t^\top Qx_t+
		\inf_{u_t\in\R^m} 
		\Bigg\{
		u_t^\top Ru_t                                      \\
		&\qquad+
		\sup_{\mu\in\mathcal P_2(\R^k)}
		\bigg[-\lambda W_2^2(\mu,\nu_t)\\
		&\qquad+\int V_{t+1}(Ax_t+Bu_t+\Xi w_t)\,\dd\mu(w_t)\bigg]
		\Bigg\},
	\end{aligned}
	\label{eq:Bellman-DP}
\end{equation}
for \(t=N-1,\ldots,0\).  

Since \(\nu_t\) has finite support, Kantorovich duality gives
\begin{equation}
	\begin{aligned}
		V_t(x_t)
		=&
		x_t^\top Qx_t+
		\inf_{u_t\in\R^m}
		\Bigg[
		u_t^\top Ru_t                                    \\
		&+
		\frac1M\sum_{i=1}^M
		\sup_{w_t\in\R^k}
		\Big\{
		V_{t+1}(Ax_t+Bu_t+\Xi w_t)\\
        & -\lambda\|w_t-\widehat  w_t^{(i)}\|^2 \Big\}
		\Bigg],
	\end{aligned}
	\label{eq:Bellman-finite-support}
\end{equation}
whenever the right-hand side is finite (see, e.g.,
\cite{Villani2009,BlanchetMurthy2019,MohajerinEsfahaniKuhn2018,GaoKleywegt2023}).

Define a matrix
\begin{equation}
	W:=BR^{-1}B^\top-\lambda^{-1}\Xi\Xi^\top .
	\label{eq:W-def-bellman}
\end{equation}

We impose the following stabilizability and observability conditions to establish Riccati convergence.

\begin{assumption}
	\label{ass:psd-W}
	$W\succeq0$, the pair \((A,W^{1/2})\) is stabilizable, and $(A,Q^{1/2})$ is observable.
\end{assumption}

\begin{remark}
    The condition \(W\succeq0\) is equivalent to $BR^{-1}B^\top\succeq \lambda^{-1}\Xi\Xi^\top$. The same matrix appeared in the Riccati analysis of \cite{KimYang2023} and was denoted by \(\Phi\). This condition imposes a structural restriction on \(B\) and \(\Xi\). In particular, it implies \(\operatorname{range}(\Xi)\subseteq \operatorname{range}(B).\) Increasing \(\lambda\) weakens this restriction.
\end{remark}


The following lemma gives the solution of the outer minimization problem and the corresponding support points of a worst-case distribution, when the value function is of quadratic form.

\begin{lemma}
	\label{lem:one-step-affine-minimax}
Fix \(t\in\{0,\ldots,N-1\}\), and suppose that at stage $t+1$, the value function \eqref{eq:value-definition} has the form
	\begin{equation}
		V_{t+1}(x_{t+1})
		=
		x_{t+1}^\top P_{t+1}x_{t+1}
		+2s_{t+1}^\top x_{t+1}
		+z_{t+1},
		\label{eq:one-step-continuation}
	\end{equation}
	where \(P_{t+1}=P_{t+1}^\top\in \R^{n\times n}\succeq0\), $s_{t+1}\in\R^n$, $z_{t+1}\in\R$.  Suppose that
	\begin{equation}
		\lambda I-\Xi^\top P_{t+1}\Xi\succ0.
		\label{eq:one-step-admissibility}
	\end{equation}
	Then the following statements hold.
	
	\begin{enumerate}
		\item The outer minimization problem in
		\eqref{eq:Bellman-finite-support} with respect to \(u\) has a
		unique solution
		\begin{equation}
			u_t^*=K_tx_t+k_t,
			\label{eq:affine-control}
		\end{equation}
		where
		\begin{align}
			K_t
			&=
			-R^{-1}B^\top
			(I+P_{t+1}W)^{-1}
			P_{t+1}A,
			\label{eq:K-affine}\\
			k_t
			&=
			-R^{-1}\!B^\top\!
			(I+P_{t+1}W)^{-1}\!
			\big(P_{t+1}\Xi d_t+s_{t+1}\big)
			\label{eq:k-affine}
		\end{align}
        with $W$ defined in \eqref{eq:W-def-bellman}.
	\item Given the control input \eqref{eq:affine-control}, the support
		point associated with the empirical sample \(\widehat w_t^{(i)}\)
		is unique and has the affine form
		\begin{equation}
			w_t^{*,i}(x)=H_tx+G_t^i,
			\qquad i=1,\ldots,M,
			\label{eq:wstar-affine-state}
		\end{equation}
		where
        \begin{equation}
		\begin{aligned}
			H_t
			&=
			(\lambda I-\Xi^\top P_{t+1}\Xi)^{-1}
			\Xi^\top P_{t+1}(A+BK_t),
			\\
			G_t^i
			&=\!
			(\lambda I\!-\!\Xi^\top\! P_{t+1}\Xi)^{-1}\!
			\Big[
			\lambda\widehat w_t^{(i)}\!+\Xi^\top\!\big(P_{t+1}Bk_t\!+\!s_{t+1}\big)
			\Big].
			\label{eq:G-affine}
		\end{aligned}
        \end{equation}
	\end{enumerate}
\end{lemma}

\begin{proof}
	The detailed proof is given in
	Appendix~\ref{app:one-step-affine-minimax}.
\end{proof}

\subsection{Finite-horizon Riccati recursion}
\label{subsec:finite-horizon-synthesis}

The one-step calculation in Lemma \ref{lem:one-step-affine-minimax} leads to the finite-horizon value representation by backward induction. The following theorem confirms that the value function for \eqref{eq:sys} and \eqref{eq:truncated-cost} indeed has the form assumed in Lemma \ref{lem:one-step-affine-minimax}. 

\begin{theorem}
	\label{thm:affine-bellman-reduction}
	Suppose that Assumptions \ref{ass:psd-W} and \ref{ass:riccati-penalty} holds.
	Then, for every \(t=0,\ldots,N\), the value function has the form
	\begin{equation}
		V_t(x_t)
		=
		x_t^\top P_tx_t+2s_t^\top x_t+z_t .
		\label{eq:affine-value}
	\end{equation}
	The coefficients satisfy the terminal conditions \(P_N=Q_f, s_N=0, z_N=0,\)
	and are computed backward in time by
	\begin{align}
		P_t&=Q+A^\top(I+P_{t+1}W)^{-1}P_{t+1}A,
		\label{eq:P-recursion-affine}\\
		s_t&=A^\top(I+P_{t+1}W)^{-1}\big(P_{t+1}\Xi d_t+s_{t+1}\big),
        \label{eq:s-recursion-affine}\\
		z_t&=z_{t+1}+\mathcal Z_t,
		\label{eq:z-recursion-affine}
	\end{align}
	where
	\begin{equation}
		\begin{aligned}
			\mathcal Z_t
			:={}&
			\frac1M\sum_{i=1}^M
			\Big[
			(Bk_t+\Xi G_t^i)^\top
			P_{t+1}
			(Bk_t+\Xi G_t^i)\\
			&
			+2s_{t+1}^\top(Bk_t+\Xi G_t^i)
			-\lambda\|G_t^i-\widehat w_t^{(i)}\|^2
			\Big]+k_t^\top Rk_t.
		\end{aligned}
		\label{eq:Z-increment}
	\end{equation}
	
	Moreover, an optimal policy pair can be obtained as follows.
	\begin{enumerate}
		\item The unique optimal control policy is given by
		\begin{equation}
			u_t^*=\pi_t^*(x_t)=K_tx_t+k_t,
			\qquad t=0,\ldots,N-1,
			\label{eq:optimal-control-policy}
		\end{equation}
		where \(K_t\) and \(k_t\) are defined in
		\eqref{eq:K-affine} and \eqref{eq:k-affine} respectively.
		
		\item Given the control policy \eqref{eq:optimal-control-policy}, the support point associated with \(\widehat w_t^{(i)}\) is $w_t^{*,i}=H_t x_t+G_t^i$, $i=1,\ldots,M$,
		where \(H_t\) and \(G_t^i\) are defined in  \eqref{eq:G-affine}. At $u_t^*$, a worst-case disturbance
		distribution policy is specified by
		\begin{equation*}
			\gamma_t^*(x_t,u_t^*)
			:=
			\frac1M\sum_{i=1}^M
			\delta_{w_t^{*,i}},
			\quad t=0,\ldots,N-1.
		\end{equation*}
	\end{enumerate}
\end{theorem}

\begin{proof}
	The detailed proof is given in
	Appendix~\ref{app:thm1}.
\end{proof}

\begin{remark}
The quadratic coefficient \(P_t\) and the matrix \(K_t\) in Theorem~\ref{thm:affine-bellman-reduction} do not depend on the
empirical mean \(d_t\). The empirical mean enters the recursion for \(s_t\)
and the formula for \(k_t\), and therefore also affects \(G_t^i\). The mean
and covariance turnpike estimates below use Riccati convergence and the
uniform product estimate.
\end{remark}

\subsection{Riccati convergence and uniform bounds}
\label{subsec:riccati-solvability}

The coefficient \(P_t\) in the finite-horizon value function is generated by the Riccati recursion \eqref{eq:P-recursion-affine}. This subsection states the conditions under which the sequence in \eqref{eq:P-recursion-affine} converges to a steady-state solution and then records a uniform lower bound over the sequence and at its limit.

The following proposition identifies the limiting stabilizing  solution associated with  \eqref{eq:P-recursion-affine}.

\begin{proposition}
	\label{prop:riccati-convergence}
	Define the Riccati map
\begin{equation}
	\mathcal R(P):=Q+A^\top(I+PW)^{-1}PA .
	\label{eq:W-R-map}
\end{equation}
with $W$ defined in \eqref{eq:W-def-bellman}. Starting from \(Q_f\), define
\begin{equation}
	S_0:=Q_f,\qquad
	S_{j+1}:=\mathcal R(S_j).\qquad j\ge0
	\label{eq:S-orbit}
\end{equation}
Then under Assumption~\ref{ass:psd-W}, the sequence \eqref{eq:S-orbit} is well defined, remains positive semidefinite and bounded, and converges to a stabilizing positive-semidefinite solution
	\(P_{\rm ss}\) of
	\(
	P=\mathcal R(P).
	\)
	Moreover,
	\begin{equation}
		A_c:=(I+WP_{\rm ss})^{-1}A
		\label{eq:Ac-def}
	\end{equation}
	is Schur.
\end{proposition}

\begin{proof}
	The detailed proof is given in
	Appendix~\ref{app:pro1}.
\end{proof}

 For a horizon \(N\), the coefficients in
\eqref{eq:P-recursion-affine} satisfy \(P_t=S_{N-t}\).
We next impose the penalty condition on the Riccati sequence and its steady-state limit. Define the Riccati tail as
\begin{align}\label{eq:closed-riccati-tail-set}
    \mathscr S:=\{P_{\rm ss}\}\cup\{S_j:j\ge0\},
\end{align}
where \(\{S_j\}_{j\ge0}\) is generated by \eqref{eq:S-orbit}.

\begin{assumption}
	\label{ass:riccati-penalty}
	$Q_f\preceq\mathcal R(Q_f)$, and the stabilizing Riccati solution satisfies
	\(
		\lambda I-\Xi^\top P_{\rm ss}\Xi\succ0 .
	\)
\end{assumption}

\begingroup
\begin{remark}
		Assumption~\ref{ass:riccati-penalty} can be checked without repeating the
penalty test for each horizon. Since
\(S_0=Q_f\preceq S_1=\mathcal R(Q_f)\) and \(\mathcal R\) is order
preserving, \(S_j\preceq S_{j+1}\) for every \(j\ge0\). Proposition~
\ref{prop:riccati-convergence} gives \(S_j\to P_{\rm ss}\), and hence
\(S_j\preceq P_{\rm ss}\). Therefore,
$\lambda I-\Xi^\top S_j\Xi
\succeq
\lambda I-\Xi^\top P_{\rm ss}\Xi
\succ0$ for all $j\ge0$.
Together with \(P_t=S_{N-t}\), this verifies the one-step penalty condition
for every horizon. The condition \(Q_f\preceq\mathcal R(Q_f)\) is used for this
horizon-independent verification. It is a sufficient condition and is not
claimed to be necessary for a fixed horizon.
\end{remark}
\endgroup

For a square matrix \(Y\), let \(\sigma_{\min}(Y)\) denote its smallest singular value. The next lemma gives a {uniform} positive lower bound over all matrices in \(\mathscr S\).

\begin{lemma}
	\label{lem:closed-tail-margin}
Under	Assumptions
   \ref{ass:psd-W} and
	\ref{ass:riccati-penalty}, there exists \(\delta>0\) such that, for
	every \(S\in\mathscr S\),
	\begin{equation}
		\begin{gathered}
			\lambda I-\Xi^\top S\Xi\succeq \delta I,\\
			\sigma_{\min}(I+SW)\ge\delta,
			\qquad
			\sigma_{\min}(I+WS)\ge\delta .
		\end{gathered}
		\label{eq:uniform-admissibility-orbit}
	\end{equation}
\end{lemma}

\begin{proof}
	The detailed proof is given in
	Appendix~\ref{app:lem2}.
\end{proof}

\section{Static reference and mean-state dynamics}
\label{sec:static}

This section defines the time-independent static mean reference and derives the finite-horizon equations for the mean state and mean adjoint, so that in Section \ref{sec:turnpike} one can show the convergence of the mean variables to the static references.

\subsection{Static Hamiltonian saddle problem}
\label{subsec:static-kkt}

{We first study the time-independent mean relations associated with
\eqref{eq:sys}. Let \(m\in\mathbb R^n\) denote the static mean state,
\(u\in\mathbb R^m\) the static control, and
\(\bar w\in\mathbb R^k\) the static mean of the distribution selected by the
adversary. They satisfy
$m=Am+Bu+\Xi\bar w$.
Here, \(\bar w\) is  not an
empirical support point.}

{Consider the Lagrangian
\begin{multline}
	\mathscr L(m,u,\bar w,p)
	:=
	m^\top Qm+u^\top Ru-\lambda\|\bar w-d\|^2 \\
	+2p^\top(Am+Bu+\Xi\bar w-m).
	\label{eq:static-lagrangian}
\end{multline}
The factor \(2\) in front of \(p^\top\) is used so that the stationarity
condition with respect to \(m\) gives
\(p=Qm+A^\top p\).}

For fixed \(m\) and \(p\), the unique minimizer with respect to \(u\) and
the unique maximizer with respect to \(\bar w\) are
\begin{equation}
	u=-R^{-1}B^\top p,
	\qquad
	\bar w=d+\lambda^{-1}\Xi^\top p .
	\label{eq:stationary-u-w}
\end{equation}
Substituting \eqref{eq:stationary-u-w} into
\eqref{eq:static-lagrangian} gives the reduced static Hamiltonian
\begin{equation*}
	\begin{aligned}
		\mathscr H_{\rm s}(m,p)
		:={}
		m^\top Qm
		+2p^\top(A-I)m-p^\top Wp
		+2p^\top\Xi d .
	\end{aligned}
\end{equation*}
We consider the reduced static saddle problem
\begin{equation}
	\min_{m\in\mathbb R^n}
	\max_{p\in\mathbb R^n}
	\mathscr H_{\rm s}(m,p).
	\label{eq:static-minimax}
\end{equation}
Since \(Q\succeq0\), \(W\succeq0\),
\(\mathscr H_{\rm s}\) is convex in \(m\) and concave in \(p\).

The stationarity conditions of
\eqref{eq:static-minimax} are
$p=Qm+A^\top p$, $m=Am-Wp+\Xi d$.
By \eqref{eq:stationary-u-w}, the second equation in
\eqref{eq:reduced-static-first-order} recovers the static mean relation
stated above. Equivalently,
\begin{equation}
	\underbrace{
		\begin{bmatrix}
			I-A & W\\
			-Q & I-A^\top
		\end{bmatrix}
	}_{\mathcal K_e}
	\begin{bmatrix}
		m\\p
	\end{bmatrix}
	=
	\begin{bmatrix}
		\Xi d\\0
	\end{bmatrix}.
	\label{eq:static-KKT-matrix}
\end{equation}

\begin{proposition}
	\label{prop:static-equilibrium}
	Under Assumptions~\ref{ass:stationary-empirical} and
	\ref{ass:psd-W}, the reduced static Hamiltonian problem
	\eqref{eq:static-minimax} has a unique saddle point
	\((m^e,p^e)\). It is characterized by
	\eqref{eq:static-KKT-matrix}. The associated static control input and
	static mean of the worst-case disturbance distribution are
	\begin{equation}
		u^e=-R^{-1}B^\top p^e,
		\qquad
		\bar w^e=d+\lambda^{-1}\Xi^\top p^e.
		\label{eq:ue-we}
	\end{equation}
	If \(d=0\), then
	$(m^e,p^e,u^e,\bar w^e)=(0,0,0,0)$.
\end{proposition}

\begin{proof}
Since \(Q\succeq0\) and \(W\succeq0\),
	\(\mathscr H_{\rm s}\) is convex in \(m\) and concave in \(p\). Its
	first-order conditions are
	\begin{equation}
		Qm+(A^\top-I)p=0,
		\qquad
		(A-I)m-Wp+\Xi d=0,
		\label{eq:reduced-static-first-order}
	\end{equation}
	which are equivalent to \eqref{eq:static-KKT-matrix}.
	
	Proposition~\ref{prop:riccati-convergence} gives the stabilizing Riccati
	solution \(P_{\rm ss}\) and the Schur matrix
	\(A_c=(I+WP_{\rm ss})^{-1}A\). The Riccati equation and the identities
	\(
	P_{\rm ss}(I+WP_{\rm ss})^{-1}
	=(I+P_{\rm ss}W)^{-1}P_{\rm ss}
	\)
	and
	\(
	I-W(I+P_{\rm ss}W)^{-1}P_{\rm ss}
	=(I+WP_{\rm ss})^{-1}
	\)
	yield
	\begin{equation*}
		\begin{aligned}
			&
			\begin{bmatrix}
				(I+WP_{\rm ss})^{-1}&0\\
				A^\top P_{\rm ss}(I+WP_{\rm ss})^{-1}&I
			\end{bmatrix}
			\mathcal K_e
			\begin{bmatrix}
				I&0\\
				P_{\rm ss}&I
			\end{bmatrix}\\
			&\qquad=
			\begin{bmatrix}
				I-A_c&(I+WP_{\rm ss})^{-1}W\\
				0&I-A_c^\top
			\end{bmatrix}.
		\end{aligned}
	\end{equation*}
	The two matrices multiplying \(\mathcal K_e\) are nonsingular. The matrix on
	the right-hand side is block upper triangular, and its diagonal blocks are
	nonsingular because \(A_c\) is Schur. Hence \(\mathcal K_e\) is nonsingular,
	and \eqref{eq:static-KKT-matrix} has a unique solution.
	
For \(m,p\in\mathbb R^n\), the first-order conditions at
	\((m^e,p^e)\) give
	\begin{align*}
		&\mathscr H_{\rm s}(m,p^e)-\mathscr H_{\rm s}(m^e,p^e)=(m-m^e)^\top Q(m-m^e)\ge0,\\
		&\mathscr H_{\rm s}(m^e,p)-\mathscr H_{\rm s}(m^e,p^e)=-(p-p^e)^\top W(p-p^e)\le0.
		\label{eq:static-saddle-max}
	\end{align*}
	Therefore, $\mathscr H_{\rm s}(m^e,p)
			\le \mathscr H_{\rm s}(m^e,p^e)\le \mathscr H_{\rm s}(m,p^e)$ for all
			$m,p\in\mathbb R^n$.
	Thus \((m^e,p^e)\) is a saddle point. Every saddle point satisfies
	\eqref{eq:reduced-static-first-order}; uniqueness follows from the
	nonsingularity of \(\mathcal K_e\). Equation~\eqref{eq:ue-we} uniquely
	determines \(u^e\) and \(\bar w^e\). If \(d=0\), uniqueness gives the zero
	solution.
\end{proof}

\begin{remark}
The reduced problem \eqref{eq:static-minimax} has the Cartesian
strategy space \(\mathbb R^n\times\mathbb R^n\). The maximizing variable
\(p\) is the static mean adjoint. The equilibrium relation
\(m=Am+Bu+\Xi\bar w\) is recovered from the first-order condition with
respect to \(p\) and \eqref{eq:stationary-u-w}. Thus, the reduced
Hamiltonian problem retains the stationary control and disturbance-mean
relations without introducing a separate zero-sum game with a shared
constraint.
\end{remark}

\begin{remark}
The empirical mean and covariance play different roles. Since the matrix
\(\mathcal K_e\) in \eqref{eq:static-KKT-matrix} is independent of \(d\) and
\(\Sigma\), the static mean reference depends linearly on \(d\). Changing
\(\Sigma\) alone does not affect this reference. By contrast, the reference
covariance defined later by \eqref{eq:Xe-Lyapunov} depends on \(\Sigma\) through
\(D(P_{\rm ss})\) and is independent of \(d\). This separates the roles of
the empirical mean and covariance in the two turnpike estimates.
\end{remark}

The tuple \((m^e,u^e,\bar w^e,p^e)\)  recovered from the
unique saddle point in Proposition~\ref{prop:static-equilibrium} is called the
static mean reference, which is generally nonzero.

\subsection{Finite-horizon mean-state and adjoint  equations}
\label{subsec:finite-horizon-mean-hamiltonian}

Next we study the dynamics of finite-horizon mean-state and the corresponding adjoint states required for the mean turnpike estimate.

Let \(x_t^*\) be the closed-loop state generated by the finite-horizon optimal control and distribution policies constructed in Theorem~\ref{thm:affine-bellman-reduction}. Let \(\mathbb E^*\) denote expectation under this optimal policy pair. For a given state \(x\), define the mean of the worst-case disturbance distribution by
\begingroup
\begin{equation}
	\gamma_t(x)
	:=
	\frac1M\sum_{i=1}^M w_t^{*,i}(x),
	\label{eq:avg-worst-case-support}
\end{equation}
\endgroup
{where \(w_t^{*,i}(\cdot)\) is the \(i\)-th support point of the worst-case disturbance distribution  in Theorem~\ref{thm:affine-bellman-reduction}.} Define the mean state, the mean control input, and the mean of the
worst-case disturbance distribution by
\begin{equation*}
		m_t^*:=\mathbb E^*[x_t^*],~
		\bar{u}_t^*:=
		\mathbb E^*\big[u_t^*(x_t^*)\big],
		~\bar{w}_t^*:=\mathbb E^*\left[\gamma_t(x_t^*)\right].
\end{equation*}
 Since the value function has the affine-quadratic form
\eqref{eq:affine-value}, define {the finite-horizon mean adjoint by}
\begin{equation}
	p_t^*:=P_tm_t^*+s_t .
	\label{eq:mean-p-affine}
\end{equation}
Here, \(m_t^*\) and \(p_t^*\) are defined for \(t=0,\ldots,N\), whereas \(\bar{u}_t^*\) and \(\bar{w}_t^*\) are defined for \(t=0,\ldots,N-1\). In what follows, we shall characterize the discrepancy between the above time-varying mean variables and the solution of the static problem \eqref{eq:static-minimax}.

We first relate the mean control input and the mean of the worst-case
disturbance distribution to the mean adjoint. We evaluate the one-step first-order optimality conditions along the finite-horizon optimal closed-loop trajectory, and average over the support points of the worst-case disturbance distribution. When Assumption~\ref{ass:stationary-empirical} holds,  taking expectations yields
\begin{align*}
    0&=\Xi^\top\big(P_{t+1}m_{t+1}^*+s_{t+1}\big)-\lambda\big(\bar{w}_t^*-d\big),\\
    0&=R\bar{u}_t^*+B^\top\big(P_{t+1}m_{t+1}^*+s_{t+1}\big).
\end{align*}
By \eqref{eq:mean-p-affine}, the above two equations yield
\begin{equation}
	\bar{u}_t^*
	=
	-R^{-1}B^\top p_{t+1}^*,
	\qquad
	\bar{w}_t^*
	=
	d+\lambda^{-1}\Xi^\top p_{t+1}^* .
	\label{eq:mean-u-w-affine}
\end{equation}

Taking expectations in the state equation \eqref{eq:sys} under the worst-case disturbance distribution and using \eqref{eq:mean-u-w-affine} gives
\begin{equation}
	m_{t+1}^*
	=
	Am_t^*
	-Wp_{t+1}^*
	+\Xi d .
	\label{eq:mean-state-hamiltonian}
\end{equation}

By \eqref{eq:mean-p-affine} and \eqref{eq:mean-state-hamiltonian},
\[
	\begin{aligned}
		p_{t+1}^*
		&=
		P_{t+1}m_{t+1}^*
		+s_{t+1} \\
		&=
		P_{t+1}
		\big(
		Am_t^*
		-Wp_{t+1}^*
		+\Xi d
		\big)
		+s_{t+1} .
	\end{aligned}
\]
Therefore,
\(
	\big(I+P_{t+1}W\big)p_{t+1}^*
	=
	P_{t+1}A m_t^*
	+P_{t+1}\Xi d
	+s_{t+1} .
\)
Since \(I+P_{t+1}W\) is nonsingular, we obtain
\begin{equation}
	\begin{aligned}
		p_{t+1}^*
		=
		&(I\!+\!P_{t+1}W)^{-1}
		\big(
		P_{t+1}A m_t^*
		\!+\!P_{t+1}\Xi d
		\!+\!s_{t+1}
		\big).
	\end{aligned}
	\label{eq:pnext-mean-expression}
\end{equation}

 From
\eqref{eq:mean-p-affine} and the coefficient recursions
\eqref{eq:P-recursion-affine}--\eqref{eq:s-recursion-affine}, we have
\begin{equation*}
	\begin{aligned}
		p_t^*
		={}&
		Qm_t^*
		+
		A^\top
		(I+P_{t+1}W)^{-1} \\
		&\quad\times
		\big(
		P_{t+1}A m_t^*
		+P_{t+1}\Xi d
		+s_{t+1}
		\big).
	\end{aligned}
\end{equation*}
Then applying \eqref{eq:pnext-mean-expression} gives
\begin{equation}
	p_t^*=Qm_t^*+A^\top p_{t+1}^* .
	\label{eq:mean-adjoint-hamiltonian}
\end{equation}

The boundary conditions of \eqref{eq:mean-state-hamiltonian} and \eqref{eq:mean-adjoint-hamiltonian}  are \(m_0^*=m_0=\mathbb E[x_0]\) and \(p_N^*=Q_fm_N^*\). The terminal condition follows from \(P_N=Q_f\), \(s_N=0\), and \eqref{eq:mean-p-affine}.

Define the errors in the mean state and mean adjoint as
\begin{equation}
	e_t:=m_t^*-m^e,
	\quad
	q_t:=p_t^*-p^e,
	\quad
	r^e:=Q_fm^e-p^e .
	\label{eq:shifted-vars}
\end{equation}
Since the static  mean reference satisfies
\[
m^e=Am^e-Wp^e+\Xi d,
\qquad
p^e=Qm^e+A^\top p^e,
\]
then subtracting these equations from \eqref{eq:mean-state-hamiltonian} and \eqref{eq:mean-adjoint-hamiltonian}, together with the boundary conditions stated above, we have
\begin{equation}
	\begin{aligned}
		e_{t+1}
		&=
		Ae_t
		-Wq_{t+1}, \quad e_0
		=
		m_0-m^e,\\
		q_t
		&=
		Qe_t
		+
		A^\top q_{t+1},\quad
		q_N
		=
		Q_fe_N+r^e .
	\end{aligned}
	\label{eq:shifted-ham-bvp}
\end{equation}

Thus, the boundary conditions in \eqref{eq:shifted-ham-bvp} are determined by the initial mean error \(m_0-m^e\) and the terminal vector \(r^e=Q_fm^e-p^e\).

\section{Mean and covariance turnpike estimates}
\label{sec:turnpike}

This section establishes the mean and covariance turnpike for the closed-loop state trajectory generated by the finite-horizon optimal policies.
The analysis proceeds in three steps.  First, exponential convergence of the Riccati sequence yields a uniform product estimate for the  finite-horizon closed-loop matrices.  Second, this estimate is applied to \eqref{eq:shifted-ham-bvp} to bound the mean variables, and then used for the covariance recursion; the resulting mean and covariance estimates are finally  used for Gaussian moment surrogates, and allows an approximation scheme for finite-horizon optimal policies.

\subsection{Riccati convergence and product estimates}
\label{subsec:riccati-tail-product}

The finite-horizon matrices \(P_t\) and the sequence \(S_j\) in \eqref{eq:S-orbit} are generated by the same Riccati map from the same terminal matrix \(Q_f\), and $P_t=S_{N-t}$. The next lemma gives an exponential convergence estimate for the sequence $P_t$. 

\begin{lemma}
	\label{lem:riccati-tail-exp}
	Under Assumption~\ref{ass:psd-W}, there exist constants \(C_P>0\) and
	\(\rho_P\in(0,1)\) such that
	\begin{equation}
		\|P_{N-j}-P_{\rm ss}\|\le C_P\rho_P^j,
		\qquad j\ge0 .
		\label{eq:riccati-exp-tail}
	\end{equation}
\end{lemma}

\begin{proof}
	Under Assumption~\ref{ass:psd-W}, Proposition~\ref{prop:riccati-convergence}
		gives \(S_j\to P_{\rm ss}\) and shows that \(A_c\) is Schur.
	Define $H(P):=(I+PW)^{-1}P$.
	Differentiating \(H\) at \(P_{\rm ss}\) and using
	\((I+P_{\rm ss}W)^{-1}P_{\rm ss}=P_{\rm ss}(I+WP_{\rm ss})^{-1}\)
	gives
	$DH(P_{\rm ss})[\Delta]=(I+P_{\rm ss}W)^{-1}\Delta(I+WP_{\rm ss})^{-1}$.
	Since \(\mathcal R(P)=Q+A^\top H(P)A\), we have
		$D\mathcal R(P_{\rm ss})[\Delta]=
		A_c^\top\Delta A_c$.
        
	Since \(A_c\) is Schur, the
	linear map \(\Delta\mapsto A_c^\top\Delta A_c\) has spectral radius smaller
	than $1$ on \(\mathbb S^n\), where \(\mathbb S^n\) is the space of real \(n\times n\) symmetric matrices. Hence there exists an equivalent norm
	\(\|\cdot\|_*\) on \(\mathbb S^n\) and \(\widehat\rho\in(0,1)\) such that
	$\|D\mathcal R(P_{\rm ss})[\Delta]\|_*
		\le
		\widehat\rho\|\Delta\|_*$.
        
	The map \(\mathcal R\) is differentiable in a neighborhood of
	\(P_{\rm ss}\). Therefore, for any \(\theta\in(\widehat\rho,1)\), there
	exists \(\varepsilon_R>0\) such that
	\begin{equation}
		\|\mathcal R(P_{\rm ss}+\Delta)-P_{\rm ss}\|_*
		\le
		\theta\|\Delta\|_*,
		\qquad
		\|\Delta\|_*\le\varepsilon_R .
		\label{eq:Riccati-local-contraction}
	\end{equation}
	Since \(S_j\to P_{\rm ss}\), there exists \(J_R\) such that
	\(
	\|S_j-P_{\rm ss}\|_*\le\varepsilon_R,
	 j\ge J_R .
	\)
	Set \(\Delta_j:=S_j-P_{\rm ss}\). Then
	\(
	\Delta_{j+1}=\mathcal R(P_{\rm ss}+\Delta_j)-P_{\rm ss},
	\)
	and \eqref{eq:Riccati-local-contraction} gives
	\(
	\|\Delta_j\|_*
	\le
	\theta^{j-J_R}\|\Delta_{J_R}\|_*,
	 j\ge J_R .
	\)
	All norms on \(\mathbb S^n\) are equivalent, and the finitely many indices
	\(0\le j<J_R\) can be absorbed into the constant. Choosing any
	\(\rho_P\in(\theta,1)\) gives \eqref{eq:riccati-exp-tail}.
\end{proof}

For \(S\in\mathscr S\), define
$A_c(S):=(I+WS)^{-1}A$. The map \(S\mapsto A_c(S)\) extends the limiting closed-loop matrix \(A_c\) to the finite-horizon Riccati matrices.  Lemma~\ref{lem:closed-tail-margin} gives uniform lower bounds for the inverses appearing in this map.  In particular, for \(S,\widetilde S\in\mathscr S\),
\begin{equation*}
		A_c(S)-A_c(\widetilde S)
		=-(I+WS)^{-1}W(S-\widetilde S)(I+W\widetilde S)^{-1}A .
\end{equation*}
Hence there exists \(L_A>0\) such that
\begin{equation}
	\|A_c(S)-A_c(\widetilde S)\|
	\le
	L_A\|S-\widetilde S\|,
	\qquad S,\widetilde S\in\mathscr S .
	\label{eq:AcS-Lipschitz}
\end{equation}

Let \(H_L=H_L^\top\succ0\) be the unique solution of $A_c^\top H_LA_c-H_L=-I$.
Such \(H_L\) exists because \(A_c\) is Schur. Since the map \(G\mapsto G^\top H_LG-H_L\) is continuous at \(A_c\), there exists \(\delta_L>0\) such that, for every \(G\in\mathbb R^{n\times n}\),
\begin{equation}
	\|G-A_c\|\le\delta_L
	\quad\Longrightarrow\quad
	G^\top H_LG-H_L\preceq-\frac12 I .
	\label{eq:common-lyap-window}
\end{equation}
Using \eqref{eq:riccati-exp-tail} and \eqref{eq:AcS-Lipschitz}, choose
\(J_L\) such that
\(
L_AC_P\rho_P^j\le\delta_L, j\ge J_L .
\)
Then every factor \(A_c(S_j)\) with \(j\ge J_L\) satisfies the Lyapunov inequality in \eqref{eq:common-lyap-window}.

{For \(0\le s\le t\le N\), define the transition matrix}
\begin{equation}
	\Phi_N(t,s):=
	\begin{cases}
		A_c(P_t)A_c(P_{t-1})\cdots A_c(P_{s+1}),
		& t>s,\\
		I, & t=s .
	\end{cases}
	\label{eq:PhiN-def}
\end{equation}
{The next lemma shows the product estimate used in both the mean and covariance turnpike analyses.}

\begin{lemma}
	\label{lem:uniform-product}
	{Under Assumptions \ref{ass:psd-W} and \ref{ass:riccati-penalty},}
	there exist constants \(C_\Phi>0\) and \(\rho\in(0,1)\), independent
	of \(N,s,t\), such that, for every \(0\le s\le t\le N\),
	\begin{equation}
		\|\Phi_N(t,s)\|
		\le
		C_\Phi\rho^{t-s} .
		\label{eq:uniform-product-bound}
	\end{equation}
\end{lemma}

\begin{proof}
	{The detailed proof is given in
	Appendix~\ref{app:lem4}.}
\end{proof}

\subsection{Mean turnpike estimate}
\label{subsec:terminal-adjoint-layer}

We now apply the product estimate \eqref{eq:uniform-product-bound} to the discrepancy estimate \eqref{eq:shifted-ham-bvp}.
Define $r_t:=q_t-P_te_t$.
At the terminal stage,
$r_N=Q_fm^e-p^e=:r^e$.

The auxiliary variable \(r_t\) satisfies a backward recursion, while \(e_t\) satisfies a forward recursion driven by \(r_t\).  This separation will be the key ingredient for the two-sided mean estimate. By definition,
\(q_{t+1}=P_{t+1}e_{t+1}+r_{t+1}.\)
Substituting this into the first equation in \eqref{eq:shifted-ham-bvp}
gives $(I+WP_{t+1})e_{t+1}=Ae_t-Wr_{t+1}$.
By Lemma~\ref{lem:closed-tail-margin}, \(I+WP_{t+1}\) is invertible. Hence
\begin{equation}
	\begin{aligned}
		e_{t+1}=
		A_c(P_{t+1})e_t
		-(I+WP_{t+1})^{-1}Wr_{t+1} .
	\end{aligned}
	\label{eq:e-driven1}
\end{equation}

Next, using the second equation in \eqref{eq:shifted-ham-bvp},
we obtain
$q_t=Qe_t+A^\top P_{t+1}e_{t+1}+A^\top r_{t+1}$.
Substituting \eqref{eq:e-driven1} yields
\begin{align*}
	q_t
	&=
	\Big[
	Q+A^\top P_{t+1}(I+WP_{t+1})^{-1}A
	\Big]e_t
	\notag\\
	&\quad+
	A^\top
	\Big[
	I-P_{t+1}(I+WP_{t+1})^{-1}W
	\Big]r_{t+1}.
\end{align*}
The coefficient of \(e_t\) is \(P_t\), since the Riccati recursion is (\ref{eq:P-recursion-affine}),
the same equality gives
$$
A^\top
	\Big[
	I-P_{t+1}(I+WP_{t+1})^{-1}W
	\Big] 
	=
	A_c(P_{t+1})^\top.
$$
Therefore,
\(
q_t
=
P_te_t+A_c(P_{t+1})^\top r_{t+1}.
\)
Subtracting \(P_te_t\) gives the backward recursion
\begin{equation}
	r_t
	=
	A_c(P_{t+1})^\top r_{t+1},
	\qquad t=0,\ldots,N-1 .
	\label{eq:r-backward}
\end{equation}

Using \eqref{eq:e-driven1} and \eqref{eq:r-backward}, the next theorem gives the two-sided mean turnpike estimate. 

\begin{theorem}
	\label{thm:affine-mean-turnpike}
	Under Assumptions~\ref{ass:stationary-empirical}, 
	\ref{ass:psd-W},
	and \ref{ass:riccati-penalty}, there exist constants
	\(C_m>0\), \(C_u>0\),  and \(\rho\in(0,1)\), independent of
	\(N\) and \(t\), such that for every \(N\ge 1\) and every
	\(t=0,\ldots,N\),
	\begin{equation}
		\begin{aligned}
			&\|m_t^*-m^e\|
			+\|p_t^*-p^e\|       \\
			&\qquad\le
			C_m
			\Big(
			\rho^t\|m_0-m^e\|
			+
			\rho^{N-t}\|Q_fm^e-p^e\|
			\Big).
		\end{aligned}
		\label{eq:two-sided-mp}
	\end{equation}
	Moreover, for \(t=0,\ldots,N-1\),
	\begin{equation}
		\begin{aligned}
			&\|\bar{u}_t^*-u^e\|
			+\|\bar{w}_t^*-\bar w^e\|       \\
			&\qquad\le
			C_u
			\Big(
			\rho^t\|m_0-m^e\|
			+
			\rho^{N-t-1}\|Q_fm^e-p^e\|
			\Big).
		\end{aligned}
		\label{eq:two-sided-uw}
	\end{equation}
\end{theorem}

\begin{proof}
	{The detailed proof is given in
	Appendix~\ref{app:thm2}.}
\end{proof}

The exponential estimate also implies a bound on the number of stage indices at which the mean state, mean control input, and mean of the
worst-case disturbance distribution lie outside a prescribed neighborhood of the static mean reference. This corresponds to the so-called measure turnpike property \cite{BerberichEtAl2018}.

\begin{corollary}
	\label{cor:measure-turnpike-affine}
	{Under Assumptions~\ref{ass:stationary-empirical},
	\ref{ass:psd-W}, 
	and \ref{ass:riccati-penalty},} {let \(C_m,C_u>0\) and \(\rho\in(0,1)\)
			be constants for which \eqref{eq:two-sided-mp}--\eqref{eq:two-sided-uw}
			hold.} For \(\varepsilon>0\) and \(a\ge0\), define
	\begin{equation*}
		L_\varepsilon(a)
		:=
		1+
		\left\lceil
		\frac{\log_+(4a/\varepsilon)}{-\log\rho}
		\right\rceil,
		\quad
		\log_+(b):=\log(\max\{b,1\}).
	\end{equation*}
	For a fixed initial mean \(m_0\), set
	\begin{equation*}
		\begin{aligned}
			\Lambda_\varepsilon(m_0)
			:={}&
			L_\varepsilon(C_m\|m_0-m^e\|)
			+
			L_\varepsilon(C_m\|r^e\|)       \\
			&+
			L_\varepsilon(C_u\|m_0-m^e\|)
			+
			L_\varepsilon(C_u\|r^e\|).
		\end{aligned}
	\end{equation*}
	Then, for all \(N\ge 1\),
	\begin{equation}
		\begin{aligned}
			\operatorname{card}\Big\{0\le t\le N-1:
			&\|m_t^*-m^e\|
			+\|\bar{u}_t^*-u^e\|        \\
			&+\|\bar{w}_t^*-\bar w^e\|
			>\varepsilon
			\Big\}
			\le
			\Lambda_\varepsilon(m_0).
		\end{aligned}
		\label{eq:measure-turnpike-cardinality}
	\end{equation}
    where $\mathrm{card}$ denotes the cardinality of a set. 
\end{corollary} 

\begin{proof}
	For \(0\le t\le N-1\), the left-hand side inside the set in	\eqref{eq:measure-turnpike-cardinality} is bounded by the sum of four
	nonnegative terms:
	\[
	\begin{array}{ll}
		T_1(t)=C_m\rho^t\|m_0-m^e\|,
		&
		T_2(t)=C_m\rho^{N-t}\|r^e\|,\\[1mm]
		T_3(t)=C_u\rho^t\|m_0-m^e\|,
		&
		T_4(t)=C_u\rho^{N-t-1}\|r^e\|.
	\end{array}
	\]
	If the sum is larger than \(\varepsilon\), then at least one of these four
	terms is larger than \(\varepsilon/4\).
	
	For a term of the form \(a\rho^t\), the inequality
	\(a\rho^t>\varepsilon/4\) holds for at most \(L_\varepsilon(a)\)
	nonnegative integers \(t\). This bounds the stage indices associated with
	\(T_1\) and \(T_3\). For the terminal terms, the changes of variables
	\(j=N-t\) and \(j=N-t-1\) give the corresponding bounds for \(T_2\) and
	\(T_4\). Adding the four bounds proves
	\eqref{eq:measure-turnpike-cardinality}.
\end{proof}

\begingroup
\begin{remark}
Corollary~\ref{cor:measure-turnpike-affine} uses the pointwise estimate in
Theorem~\ref{thm:affine-mean-turnpike} to obtain an upper bound on the number
of off-turnpike stages which is
independent of \(N\). Hence, increasing \(N\) enlarges the interior part of
the horizon without increasing the bound on the number of off-turnpike
stages.
\end{remark}
\endgroup

\subsection{Covariance turnpike}
\label{subsec:covariance-turnpike}

Define the closed-loop state covariance by
\begin{equation*}
    X_t^*:=\mathbb E^*\big[(x_t^*-m_t^*)(x_t^*-m_t^*)^\top\big].
\end{equation*}
Under
Assumption~\ref{ass:stationary-empirical}, \(X_0\) is given by
\eqref{eq:initial-ensemble-moments}. For a deterministic initial state, \(X_0=0\).

Under the optimal policy pair and conditional on
\(x_t^*=x\), the distribution of \(x_{t+1}^*\) assigns mass \(1/M\) to each of the points 
$x_{t+1}^*=Ax+B u_t^*(x)+\Xi w_t^{*,i}(x)$,
for \( i=1,\ldots,M\). Hence the conditional mean of $x_{t+1}^*$ is
\begin{equation}
	\begin{aligned}
		\mathbb E^*
		\big[
		x_{t+1}^*
		\big| x_t^*=x
		\big] =
		Ax+B u_t^*(x)
		+\Xi\gamma_t(x).
	\end{aligned}
	\label{eq:psi-def}
\end{equation}



We next identify the linear part of this conditional mean. Substituting the optimal control \eqref{eq:affine-control} and the worst-case support points \eqref{eq:wstar-affine-state} into \eqref{eq:psi-def} gives
\[
\begin{aligned}
	&\mathbb E^*
	\big[
	x_{t+1}^*
	\big| x_t^*=x
	\big] \\
    =&
	Ax-W(I\!+\!P_{t+1}W)^{-1}
	\big(
	P_{t+1}Ax
	+P_{t+1}\Xi d
	+s_{t+1}
	\big)
	+\Xi d .
\end{aligned}
\]
Therefore, for any \(x,y\in\mathbb R^n\),
\begin{equation*}
		\mathbb E^*
		\big[
		x_{t+1}^*
		\big| x_t^*=x
		\big]
		-
		\mathbb E^*
		\big[
		x_{t+1}^*
		\big| x_t^*=y
		\big]=
		A_c(P_{t+1})(x-y).
\end{equation*}
Here we used
\(
I-W(I+P_{t+1}W)^{-1}P_{t+1}
=
(I+WP_{t+1})^{-1}.
\)

Since the conditional mean is affine in the current state, evaluating
it at \(x_t^*\) and \(m_t^*\), respectively, gives
\begin{equation}
	\begin{aligned}
		&\mathbb E^*
		\big[
		x_{t+1}^*
		\big| x_t^*
		\big]
		-
		m_{t+1}^*  =
		A_c(P_{t+1})
		(x_t^*-m_t^*).
	\end{aligned}
	\label{eq:conditional-mean-centered}
\end{equation}

{We now compute the conditional centered second moment.}
For fixed \(x\), it follows from  (\ref{eq:psi-def}) that
\[
x_{t+1}^*-\mathbb E^*
		\big[
		x_{t+1}^*
		\big| x_t^*=x
		\big] =
	\Xi
	\big(
	w_t^{*,i}(x)-\gamma_t(x)
	\big).
\]
Thus, by the definition of conditional covariance,
\begin{equation}
	\begin{aligned}
		\operatorname{Cov}^{*}
		\big(
		x_{t+1}^*
		\big|
		x_t^*=x
		\big) 
		:=&
		\mathbb E^*
		\Big[
		\big(
		x_{t+1}^*
		-
		\mathbb E^*
		[
		x_{t+1}^*
		\big|
		x_t^*=x
		]
		\big) \\
		&\!\!\!\!\!\!\!\!\!\!\!\!\!\!\!\!\!\!\!\!\!\!\!\!\!\!\!\!\!\!\!\!\!\times
		\big(
		x_{t+1}^*
		-
		\mathbb E^*
		[
		x_{t+1}^*
		\big|
		x_t^*=x
		]
		\big)^\top
		\,\Big|\,
		x_t^*=x
		\Big] \\
		&\!\!\!\!\!\!\!\!\!\!\!\!\!\!\!\!\!\!\!\!\!\!\!\!\!\!\!\!\!\!\!\!\!\!\!\!\!\!\!\!\!\!\!\!\!\!\!\!\!\!\!\!\!\!\!\!\!\!\!\!\!\!=
		\frac1M\sum_{i=1}^M
		\Xi
		\big(w_t^{*,i}(x)-\gamma_t(x)\big)\big(w_t^{*,i}(x)-\gamma_t(x)\big)^\top
		\Xi^\top .
	\end{aligned}
	\label{eq:eta-conditional-cov}
\end{equation}
By Assumption~\ref{ass:stationary-empirical},
\(
\frac1M\sum_{i=1}^M(
\widehat w_t^{(i)}-d)=0,
\frac1M\sum_{i=1}^M
(\widehat w_t^{(i)}-d)(\widehat w_t^{(i)}-d)^\top
=
\Sigma .
\)
Define
\begin{align}\label{eq:DP}
	D(P)
	:={}&
	\lambda^2
	\Xi
	(\lambda I-\Xi^\top P\Xi)^{-1}
	\Sigma 
	(\lambda I-\Xi^\top P\Xi)^{-1}
	\Xi^\top,
\end{align}
then the centered support points have zero mean and the induced covariance satisfies
\begingroup
$$
\frac1M\sum_{i=1}^M
		\Xi\big(w_t^{*,i}(x)-\gamma_t(x)\big)
		\big(w_t^{*,i}(x)-\gamma_t(x)\big)^\top\Xi^\top=D(P_{t+1}).
$$
We now derive the covariance recursion.
Write
\begin{equation*}
	\begin{aligned}
		&x_{t+1}^*-m_{t+1}^*\\
		=&
		\Big(
		\mathbb E^*
		[
		x_{t+1}^*
		\big|
		x_t^*
		]
		-
		m_{t+1}^*
		\Big) +
		\Big(
		x_{t+1}^*
		-
		\mathbb E^*
		[
		x_{t+1}^*
		\big|
		x_t^*
		]
		\Big),
	\end{aligned}
\end{equation*}
in which the first term is measurable with respect to \(x_t^*\), while the second term has zero conditional mean given \(x_t^*\). Hence the cross terms vanish after taking expectations.


By \eqref{eq:conditional-mean-centered}, the contribution of the centered
conditional mean is
\[
\begin{aligned}
	&\mathbb E^*\Big[\big(\mathbb E^*[x_{t+1}^*\big|x_t^*]-m_{t+1}^*\big)\big(\mathbb E^*[x_{t+1}^*\big|x_t^*]-m_{t+1}^*\big)^\top\Big] \\
	=&\mathbb E^*\Big[A_c(P_{t+1})(x_t^*-m_t^*) (x_t^*-m_t^*)^\top A_c(P_{t+1})^\top\Big] \\
	=&A_c(P_{t+1})X_t^*A_c(P_{t+1})^\top .
\end{aligned}
\]
The contribution of the conditional fluctuation is
\[
\begin{aligned}
	&\mathbb E^*
	\Big[\big(x_{t+1}^*-\mathbb E^*[x_{t+1}^*\big|x_t^*]\big)
	\big(x_{t+1}^*-	\mathbb E^*[x_{t+1}^*\big|x_t^*]\big)^\top\Big] \\
	=&\mathbb E^*\left[\operatorname{Cov}^{*}\big(x_{t+1}^*\big|x_t^*\big)\right] \\
	=&D(P_{t+1}),
\end{aligned}
\]
where the last equality uses \eqref{eq:eta-conditional-cov}. Combining the two gives
\begin{equation}
	X_{t+1}^*
	=
	A_c(P_{t+1})
	X_t^*
	A_c(P_{t+1})^\top
	+
	D(P_{t+1}).
	\label{eq:cov-recursion}
\end{equation}

Replacing \(P_{t+1}\) by its limit \(P_{\rm ss}\) in \eqref{eq:cov-recursion} gives the reference covariance.  Define \(X^e\) by
\begin{equation}
	X^e
	=
	A_cX^eA_c^\top+D(P_{\rm ss}),
	\qquad
	A_c=(I+WP_{\rm ss})^{-1}A .
	\label{eq:Xe-Lyapunov}
\end{equation}
Since \(A_c\) is Schur, this equation has a unique symmetric positive semidefinite solution.

The next theorem gives a two-sided covariance turnpike estimate. The
initial term is determined by \(X_0-X^e\), whereas the terminal term reflects
the deviation of \(P_{t+1}\) from \(P_{\rm ss}\) near the terminal stage.
\begin{theorem}
	\label{thm:cov-turnpike}
	Under Assumptions~\ref{ass:stationary-empirical}, \ref{ass:psd-W}, and
	\ref{ass:riccati-penalty}, there exist constants \(C_X>0\)
	 and \(\rho_X\in(0,1)\), independent of \(N\) and \(t\), such
	that for every \(N\ge 1\) and every \(t=0,\ldots,N\),
	\begin{equation}
		\|X_t^*-X^e\|
		\le
		C_X
		\left(
		\rho_X^t\|X_0-X^e\|
		+
		\rho_X^{N-t}
		\right).
		\label{eq:cov-turnpike-bound}
	\end{equation}
\end{theorem}

\begin{proof}
	{The detailed proof is given in
	Appendix~\ref{app:thm3}.}
\end{proof}

\begingroup
\begin{remark}
No terminal condition is imposed on the covariance recursion. Nevertheless,
the covariance estimate contains a terminal term because \(P_{t+1}\) differs
from \(P_{\rm ss}\) near the terminal stage. Consequently,
\(A_c(P_{t+1})\) differs from \(A_c\), and \(D(P_{t+1})\) differs from
\(D(P_{\rm ss})\). This terminal term therefore comes from the
finite-horizon Riccati matrices. Establishing the covariance turnpike requires a separate analysis of the covariance recursion and does not follow
directly from the mean turnpike estimate.
\end{remark}
\endgroup

\subsection{Gaussian moment surrogates}
\label{subsec:second-order-distributional}

As consequences of the mean and covariance turnpike estimates, one can derive bounds for quadratic functions of the closed-loop state and a 2-Wasserstein estimate for Gaussian moment surrogates defined by the same moments.

Fix a quadratic function 
\begin{align}\label{eq:quadratic}
    \varphi(x):=x^\top\varPhi x+2h^\top x+c,
\end{align}
where
\(\varPhi=\varPhi^\top\in\mathbb R^{n\times n}\),
\(h\in\mathbb R^n\), and \(c\in\mathbb R\) are fixed.
Define the corresponding reference value by
$\varphi^e:=(m^e)^\top\varPhi m^e
	+\operatorname{tr}(\varPhi X^e)
	+2h^\top m^e+c$.
The following proposition gives the corresponding turnpike estimate for
\(\mathbb E^*[\varphi(x_t^*)]\).

\begin{proposition}
	\label{prop:quadratic-observable-turnpike}
	Under Assumptions~\ref{ass:stationary-empirical}, \ref{ass:psd-W}, 
	and \ref{ass:riccati-penalty}, fix \(R_0>0\)
	and consider initial state distributions satisfying
	\begin{equation}
		\|m_0-m^e\|+\|X_0-X^e\|\le R_0 .
		\label{eq:initial-moment-R0-bound}
	\end{equation}
	Then for quadratic function \eqref{eq:quadratic}, there exist constants \(C_{\varphi,R_0}>0\)
	 and \(\rho_\varphi\in(0,1)\), independent of \(N,t\) and of
	the particular initial distribution satisfying \eqref{eq:initial-moment-R0-bound},
	such that for every \(N\ge 1\) and \(t=0,\ldots,N\),
	\begin{equation}
		\begin{aligned}
			&\left|
			\mathbb E^*[\varphi(x_t^*)]-\varphi^e
			\right|                                      \\
			\le&~
			C_{\varphi,R_0}
			\left[
			\rho_\varphi^t
			\big(
			\|m_0-m^e\|+\|X_0-X^e\|
			\big)
			+\rho_\varphi^{N-t}
			\right].
		\end{aligned}
		\label{eq:quadratic-observable-turnpike-bound}
	\end{equation}
\end{proposition}

\begin{proof}
	Since \(m_t^*=\mathbb E^*[x_t^*]\) and \(X_t^*\) is the
	covariance matrix,
	$\mathbb E^*[\varphi(x_t^*)]
			=
			(m_t^*)^\top \varPhi m_t^*
			+\operatorname{tr}(\varPhi X_t^*)
			+2h^\top m_t^*+c$.
	Subtracting $\varphi^e$ gives
	\begin{equation}
		\begin{aligned}
			\mathbb E^*[\varphi(x_t^*)]-\varphi^e
			={}&
			(m_t^*-m^e)^\top \varPhi (m_t^*+m^e)\\
			&\!\!\!\!\!\!\!\!\!\!+
			\operatorname{tr}\!\big(\varPhi (X_t^*-X^e)\big)
			+2h^\top(m_t^*-m^e).
		\end{aligned}
		\label{eq:quadratic-observable-difference-expansion}
	\end{equation}
	Based on the estimates \eqref{eq:two-sided-mp} and \eqref{eq:cov-turnpike-bound} in Theorem \ref{thm:affine-mean-turnpike} and \ref{thm:cov-turnpike},
	on the bounded set \eqref{eq:initial-moment-R0-bound} the means are
	uniformly bounded as
	\(
	\sup_{N,t}\|m_t^*\|\le B_{R_0}
	\)
	for some \(B_{R_0}<\infty\). Hence
	\eqref{eq:quadratic-observable-difference-expansion} gives
	\[
	\begin{aligned}
		\left|
		\mathbb E^*[\varphi(x_t^*)]-\varphi^e
		\right|        \le
		C_{\varPhi,h,R_0}
		\left(
		\|m_t^*-m^e\|
		+
		\|X_t^*-X^e\|
		\right).
	\end{aligned}
	\]
	Choosing any \(\rho_\varphi\in(\max\{\rho,\rho_X\},1)\) and increasing
	the constant gives \eqref{eq:quadratic-observable-turnpike-bound}.
\end{proof}

Let \(\mathcal N(m,X)\) denote the Gaussian probability measure on
\(\mathbb R^n\) with mean \(m\in\mathbb R^n\) and covariance
\(X\succeq0\). Define
\begin{equation*}
    \mathcal G_t
    :=
    \mathcal N(m_t^*,X_t^*),
    \qquad
    \mathcal G^e
    :=
    \mathcal N(m^e,X^e).
\end{equation*}
These probability measures are Gaussian moment surrogates that match the
finite-horizon and reference means and covariances, respectively. The actual
closed-loop state distributions are not assumed to be Gaussian. The preceding
mean and covariance turnpike estimates imply the following Wasserstein bound.

\begin{corollary}
    \label{cor:gaussian-surrogate-turnpike}
    Under Assumptions~\ref{ass:stationary-empirical},
    \ref{ass:psd-W}, and \ref{ass:riccati-penalty}, there exist constants
    \(C_G>0\) and \(\rho_G\in(0,1)\), independent of \(N\), \(t\), and the
    particular initial distribution, such that, for every \(N\ge1\) and
    \(t=0,\ldots,N\),
    \begin{equation}
        \begin{aligned}
            &W_2\bigl(\mathcal G_t,\mathcal G^e\bigr)\\
            \leq&
            C_G\Big[
                \rho_G^t
                \bigl(
                    \|m_0-m^e\|
                    +
                    \|X_0-X^e\|^{1/2}
                \bigr)+
                \rho_G^{N-t}
            \Big].
        \end{aligned}
        \label{eq:gaussian-surrogate-W2-bound}
    \end{equation}
\end{corollary}

\begin{proof}
	{The detailed proof is given in
	Appendix~\ref{app:cro1}.}
\end{proof}

\subsection{Interior approximation of the mean and covariance}
\label{subsec:interior-static-approximation}

The preceding turnpike estimates give a uniform approximation of the optimal
mean variables and covariance over the interior of the horizon. For an integer
\(L\ge0\) and a horizon \(N\ge2L+1\), define $\mathcal I_{N,L}
	:=
	\{L,\ldots,N-L-1\}$, which is obtained by removing \(L\) stages from each end
of the horizon.

Let \(C_m,C_u,\rho\) and \(C_X,\rho_X\) be the constants in
Theorems~\ref{thm:affine-mean-turnpike} and
\ref{thm:cov-turnpike}, respectively. Recall that
\(r^e=Q_fm^e-p^e\), and set
\begin{equation*}
	\begin{aligned}
		\rho_{\rm int}
		&:=
		\max\{\rho,\rho_X\},\\
		C_{\rm int}
		&:=\!
		(C_m\!+\!C_u)
		\bigl(
		\|m_0\!-\!m^e\|\!+\!\|r^e\|
		\bigr)\!+\!
		C\!_X
		\bigl(
		\|X_0\!-\!X^e\|\!+\!1
		\bigr).
	\end{aligned}
\end{equation*}

\begin{corollary}
	\label{cor:interior-static-approximation}
	Under Assumptions~\ref{ass:stationary-empirical}, \ref{ass:psd-W}, and \ref{ass:riccati-penalty},  for every integer \(L\ge0\), every
	\(N\ge2L+1\), and every \(t\in\mathcal I_{N,L}\),
	\begin{equation*}
			\|e_t\|
			\!+\!\|q_t\|\!+\!
			\|\bar{u}_t^*-u^e\|
			\!+\!\|\bar{w}_t^*-\bar w^e\|\!+\!
			\|X_t^*-X^e\|
			\le
			C_{\rm int}\rho_{\rm int}^{L}.
	\end{equation*}
\end{corollary}

\begin{proof}
	For \(0\le t\le N-1\), by Theorems
	\ref{thm:affine-mean-turnpike} and
	\ref{thm:cov-turnpike} and the fact that
	\(\rho\le\rho_{\rm int}\) and
	\(\rho_X\le\rho_{\rm int}\), we have
	\[
	\begin{aligned}
		&\|e_t\|
			+\|q_t\|+
			\|\bar{u}_t^*-u^e\|
			+\|\bar{w}_t^*-\bar w^e\|+
			\|X_t^*-X^e\|\\
		\le&
		\Big[
		(C_m+C_u)\|m_0-m^e\|
		+C_X\|X_0-X^e\|
		\Big]
		\rho_{\rm int}^{t}\\
		&\quad+
		\Big[
		(C_m+C_u)\|r^e\|
		+C_X
		\Big]
		\rho_{\rm int}^{N-t-1}.
	\end{aligned}
	\]
	Here the terms with exponent \(N-t\) have been bounded by terms with exponent
	\(N-t-1\). If \(t\in\mathcal I_{N,L}\), then
	\(t\ge L\) and \(N-t-1\ge L\). Hence both exponential factors are bounded by
	\(\rho_{\rm int}^{L}\), which proves
	the corollary.
\end{proof}

For a given tolerance \(\varepsilon>0\), let
\(L_{\varepsilon}^{\rm int}\) be the smallest nonnegative integer such that
$C_{\rm int}\rho_{\rm int}^{L_{\varepsilon}^{\rm int}}
\leq \varepsilon$.
Equivalently,
\begin{equation*}
	L_{\varepsilon}^{\rm int}
	=
	\left\lceil
	\frac{
		\log\!\left(
		\max\{C_{\rm int}/\varepsilon,1\}
		\right)
	}{
		-\log\rho_{\rm int}
	}
	\right\rceil .
\end{equation*}
If \(N\geq2L_{\varepsilon}^{\rm int}+1\), then Corollary \ref{cor:interior-static-approximation} gives
$\|e_t\|+\|q_t\|+\|\bar{u}_t^*-u^e\|+\|\bar{w}_t^*-\bar w^e\|+\|X_t^*-X^e\|\leq\varepsilon$
for every
\(t\in\mathcal I_{N,L_{\varepsilon}^{\rm int}}\).
Thus, the finite-horizon values may be retained on the initial and terminal
boundary layers, while the corresponding time-independent references are
used over the interior of the horizon. The resulting error, measured by the
sum above, is at most \(\varepsilon\).

Moreover,
$\operatorname{card}\bigl(\mathcal I_{N,L_{\varepsilon}^{\rm int}}\bigr)=N-2L_{\varepsilon}^{\rm int}$.
Since \(L_{\varepsilon}^{\rm int}\) is independent of \(N\),
$\lim\limits_{N\rightarrow\infty}\frac{\operatorname{card}\bigl(\mathcal I_{N,L_{\varepsilon}^{\rm int}}\bigr)}{N}
    =1$.

\endgroup

\subsection{Approximation using quasi-turnpike with error estimation}\label{subsec:quasi-turnpike}

Corollary~\ref{cor:interior-static-approximation} concerns only the optimal mean variables and covariance, and does not construct an
approximate control and disturbance distribution policy pair or establish a finite-horizon performance guarantee for such a pair. In this subsection we show that one can combine static feedback and finite horizon optimal steering to approximate the optimal control on the whole finite horizon.

As established in Section \ref{sec:bellman-riccati}, the finite-horizon optimal control policy is \(u_t^*=K_tx_t+k_t\). We first
define the time-independent coefficients
\begin{align*}
K_{\rm ss}&=-R^{-1}B^\top(I+P_{\rm ss}W)^{-1}P_{\rm ss}A,\\
k_{\rm ss}
&=-R^{-1}B^\top(I+P_{\rm ss}W)^{-1}
\bigl(P_{\rm ss}\Xi d+p^e-P_{\rm ss}m^e\bigr).
\end{align*}
They are obtained from \eqref{eq:K-affine}--\eqref{eq:k-affine} by replacing
\(P_{t+1}\) with \(P_{\rm ss}\) and \(s_{t+1}\) with
\(p^e-P_{\rm ss}m^e\). The associated time-independent affine control policy is
\(K_{\rm ss}x+k_{\rm ss}\).

For \(q\in\{0,\ldots,N\}\), define the hybrid control policy
\begin{equation}\label{eq:quasi-turnpike}
\pi_t^{(q)}(x)=
\begin{cases}
K_{\rm ss}x+k_{\rm ss},&0\le t<N-q,\\
K_tx+k_t,&N-q\le t\le N-1.
\end{cases}
\end{equation}
Here, \(q\) is the number of terminal stages at which the finite-horizon
coefficients are retained. Thus, \(q=0\) gives the time-independent affine control policy,
whereas \(q=N\) gives the finite-horizon optimal control policy.

For a fixed control policy \(\pi\), let \(J_N^{\rm wc}(\pi)\) denote the value
of \eqref{eq:truncated-cost}, averaged over the prescribed initial
distribution and maximized over \(\gamma\in\Gamma_{0:N-1}\). Let
\(J_N^*:=J_N^{\rm wc}(\pi^*)\), where \(\pi^*\) is the finite-horizon optimal control. The following proposition quantifies the approximation error of the costs of $\pi^{(q)}$ and $\pi^*$.

\begin{proposition}[Performance of the hybrid policy]\label{prp:quasi-turnpike-cost}
	Suppose that Assumptions~\ref{ass:stationary-empirical}, \ref{ass:psd-W} and \ref{ass:riccati-penalty} hold, and consider the policy \eqref{eq:quasi-turnpike}.
	Then there exist $q_0\in\mathbb{N}$, $C_{\rm p}>0$, and
	$\rho_{\rm p}\in(0,1)$, independent of $N$ and $q$, such that, for
	every $N\ge q\ge q_0$,
    \begin{align}\label{71}
     0\le \!J_N^{\rm wc}(\pi^{(q)})\!-\!J_N^\ast\le \!C_{\rm p}\!\left(\!1\!+\!\int_{\mathbb{R}^n}\!\!\!\|x\|^2\dd\mu_0(x)\!\right)\!\rho_{\rm p}^{\,2q}.   
    \end{align}
\end{proposition}

\begin{proof}
The detailed proof is given in
Appendix~\ref{app:prop4}.
\end{proof}

This proposition shows that the turnpike structure can be used to construct an approximate control policy with an explicit horizon-uniform bound on its worst-case finite-horizon performance loss, which in practice reduces computational complexity by replacing part of time-varying policy by static controls. In the literature such strategy is sometimes referred to as quasi-turnpike method \cite{TrelatZuazua2015, Esteve2022}.

\section{Numerical Experiments}
\label{sec:numerical-experiments}

\subsection{Two-dimensional example and empirical support}

We first consider the two-dimensional system
with $A=\left[\begin{matrix}
1.05 & 0.12\\
0    & 0.92
\end{matrix}\right]$,
$B=I_2$, $\Xi=0.35I_2$, $Q=\left[\begin{matrix}
1&0\\0&0.7
\end{matrix}\right]$, $Q_f=Q$, $R=0.8I_2$, $\lambda=4$.
The horizon is \(N=80\). The empirical and initial moments are
$d=\begin{bmatrix}0.20\\-0.12\end{bmatrix}$, $\Sigma=\left[\begin{matrix}
0.050&0.012\\0.012&0.035
\end{matrix}\right]$,
$m_0=\begin{bmatrix}2.0\\-1.5\end{bmatrix}$,
$X_0=\left[\begin{matrix}0.25&0\\0&0.18\end{matrix}\right]$.

The empirical distribution has \(M=12\) equally weighted support points. We
construct them as
\begin{equation*}
\widehat w^{(i)}
=d+\sqrt{2}\,\Sigma^{1/2}
\begin{bmatrix}
\cos\theta_i\\ \sin\theta_i
\end{bmatrix},
\qquad
\theta_i=\frac{2\pi(i-1)}{M},
\end{equation*}
where \(\theta_i\) gives \(M\) equally spaced directions and
\(\Sigma^{1/2}\) is the symmetric positive-semidefinite square root. Hence,
\[
\begin{aligned}
\frac1M\sum_{i=1}^M\widehat w^{(i)}=d,~
\frac1M\sum_{i=1}^M
(\widehat w^{(i)}-d)(\widehat w^{(i)}-d)^\top=\Sigma.
\end{aligned}
\]
The same support is used at every stage. Table~\ref{tab:numerical-support}
lists the support points.

\begin{table}[!t]
\centering
\caption{Equal-weight empirical support used in the
two-dimensional example}
\label{tab:numerical-support}
\resizebox{\columnwidth}{!}{%
\begin{tabular}{c rr@{\hspace{5mm}}c rr}
\toprule
\(i\)&\(\widehat w_1^{(i)}\)&\(\widehat w_2^{(i)}\)&
\(i\)&\(\widehat w_1^{(i)}\)&\(\widehat w_2^{(i)}\)\\
\midrule
1  &  0.513458 & -0.078240 & 7  & -0.113458 & -0.161760\\
2  &  0.492343 &  0.046794 & 8  & -0.092343 & -0.286794\\
3  &  0.392894 &  0.127137 & 9  &  0.007106 & -0.367137\\
4  &  0.241760 &  0.141259 & 10 &  0.158240 & -0.381259\\
5  &  0.079436 &  0.085377 & 11 &  0.320564 & -0.325377\\
6  & -0.050583 & -0.025536 & 12 &  0.450583 & -0.214464\\
\bottomrule
\end{tabular}}
\end{table}

For these data, $W=1.219375I_2\succ0$, $\lambda_{\min}
\bigl(\lambda I-\Xi^\top P_{\rm ss}\Xi\bigr)
\approx3.8028$.
The spectral radius of \(A_c\) is approximately \(0.3735\), and
\(\sigma_{\min}(\mathcal K_e)\approx0.7107\). The reduced static Hamiltonian saddle problem gives
\begin{small}
    \[
m^e\!\!=\!\!\begin{bmatrix}-0.0028\\-0.0134\end{bmatrix}\!,
p^e\!\!=\!\!\begin{bmatrix}0.0560\\-0.0336\end{bmatrix}\!,
u^e\!\!=\!\!\begin{bmatrix}-0.0700\\0.0420\end{bmatrix}\!,
\bar w^e\!\!=\!\!\begin{bmatrix}0.2049\\-0.1229\end{bmatrix}\!.
\]
\end{small}
We repeated the computation with other equally weighted supports having the
same \(d\) and \(\Sigma\). The resulting mean and covariance sequences agreed
to numerical precision. This is consistent with the recursions, which depend
on the empirical support through its first two moments.

\subsection{Turnpike estimates and interior approximation}

All vector errors use the Euclidean norm, and all matrix errors use the
Frobenius norm. The dashed curves in
Figures~\ref{fig:numerical-mean-adjoint}--\ref{fig:numerical-gaussian} are
a posteriori numerical envelopes. They are constructed after the
finite-horizon errors have been computed. The decay rates are chosen from the
computed Riccati convergence and the contraction of the limiting closed-loop
matrix. The multiplicative coefficients are then selected to cover the errors
above the numerical display threshold. These envelopes illustrate the
corresponding two-sided exponential form. They are not evaluations of the
abstract constants in the analytical estimates.

\begin{figure}[!t]
\centering
\subfloat[]{%
\includegraphics[width=0.485\linewidth]{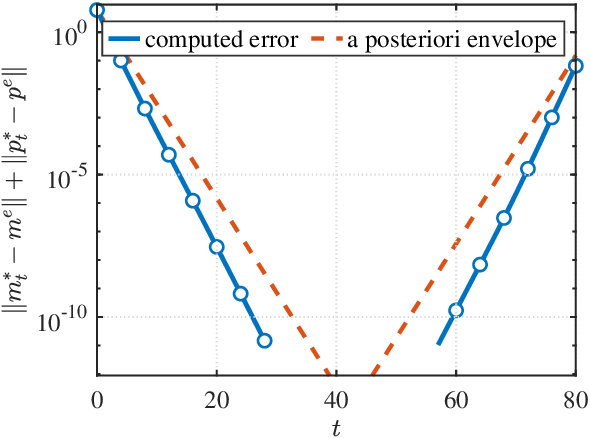}%
\label{fig:numerical-mean-adjoint}}
\hfill
\subfloat[]{%
\includegraphics[width=0.485\linewidth]{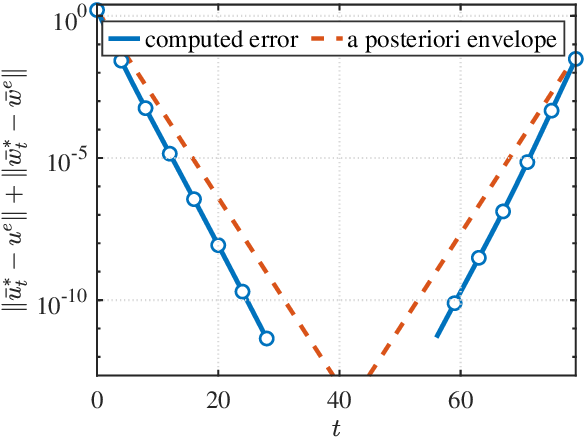}%
\label{fig:numerical-input-disturbance}}
\caption{Mean turnpike errors. (a)
\(\|m_t^*-m^e\|+\|p_t^*-p^e\|\). (b)
\(\|\bar u_t^*-u^e\|+\|\bar w_t^*-\bar w^e\|\). The dashed curves are
the corresponding a posteriori numerical envelopes.}
\label{fig:numerical-mean-turnpike}
\end{figure}

\begin{figure}[!t]
\centering
\subfloat[]{%
\includegraphics[width=0.485\linewidth]{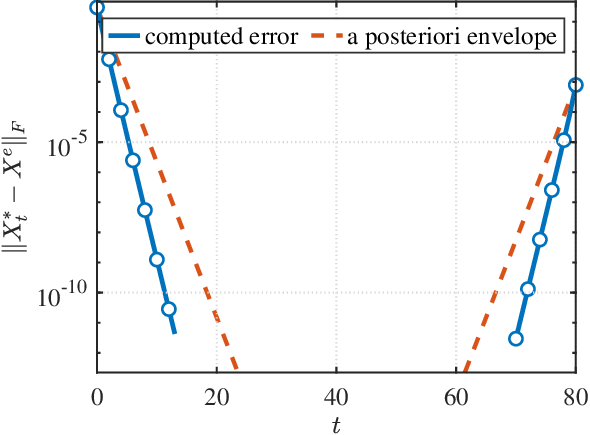}%
\label{fig:numerical-covariance}}
\hfill
\subfloat[]{%
\includegraphics[width=0.485\linewidth]{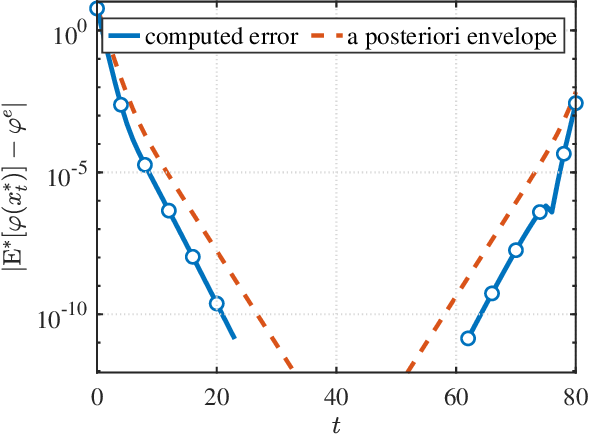}%
\label{fig:numerical-quadratic-observable}}
\caption{Covariance and quadratic-function errors. (a)
\(\|X_t^*-X^e\|_F\). (b)
\(\left|\mathbb E^*[\varphi(x_t^*)]-\varphi^e\right|\) for
\(\varphi(x)=x^\top Qx\). The dashed curves are the corresponding a
posteriori numerical envelopes.}
\label{fig:numerical-covariance-quadratic}
\end{figure}

Figures~\ref{fig:numerical-mean-adjoint} and
\ref{fig:numerical-input-disturbance} plot the left-hand sides of
\eqref{eq:two-sided-mp} and \eqref{eq:two-sided-uw}, respectively.
Figure~\ref{fig:numerical-covariance} plots the left-hand side of
\eqref{eq:cov-turnpike-bound}. Figure~\ref{fig:numerical-quadratic-observable}
corresponds to \eqref{eq:quadratic-observable-turnpike-bound}.

For the measure-turnpike plot, a stage \(t\) is called off-turnpike when $\|m_t^*-m^e\|+\|\bar u_t^*-u^e\|
+\|\bar w_t^*-\bar w^e\|>\varepsilon$.
The plotted fraction is the number of such stages divided by \(N\).

\begin{figure}[!t]
\centering
\subfloat[]{%
\includegraphics[width=0.485\linewidth]{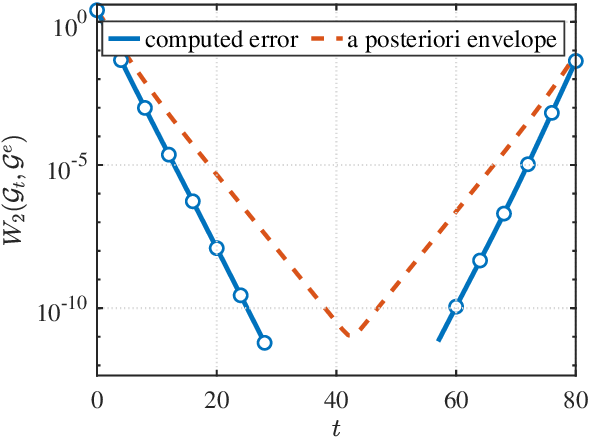}%
\label{fig:numerical-gaussian}}
\hfill
\subfloat[]{%
\includegraphics[width=0.485\linewidth]{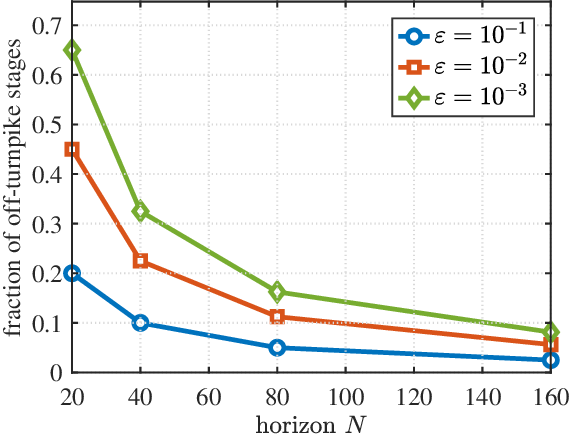}%
\label{fig:numerical-measure-turnpike}}
\caption{Gaussian moment-surrogate and measure-turnpike. (a) \(W_2(\mathcal G_t,\mathcal G^e)\); the dashed curve is the
corresponding a posteriori numerical envelope. (b) Fraction of off-turnpike stages for three tolerances and four
horizons.}
\label{fig:numerical-gaussian-measure}
\end{figure}

Figure~\ref{fig:numerical-gaussian} corresponds to
\eqref{eq:gaussian-surrogate-W2-bound}. The probability measures
\(\mathcal G_t\) and \(\mathcal G^e\) are the Gaussian surrogates
defined before Corollary~\ref{cor:gaussian-surrogate-turnpike} (which are not
the actual closed-loop state distributions). The decrease in
Figure~\ref{fig:numerical-measure-turnpike} agrees with
Corollary~\ref{cor:measure-turnpike-affine}, which bounds the number of
off-turnpike stages independently of \(N\).

For the interior approximation, define
$\eta_t:=
\|m_t^*-m^e\|+\|p_t^*-p^e\|
+\|\bar u_t^*-u^e\|+\|\bar w_t^*-\bar w^e\|+\|X_t^*-X^e\|_F$
and $E_{N,L}:=\max_{t\in\mathcal I_{N,L}}\eta_t$. Then \(\eta_t\) is the sum of the errors in
Corollary~\ref{cor:interior-static-approximation}, and \(E_{N,L}\) is their
largest value on \(\mathcal I_{N,L}\).

\begin{figure}[!t]
\centering
\subfloat[]{%
\includegraphics[width=0.485\linewidth]{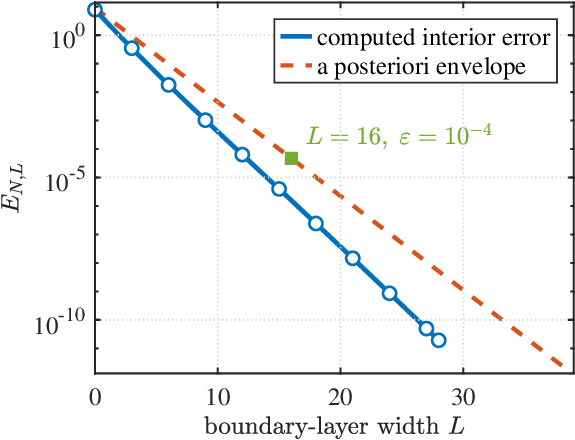}%
\label{fig:numerical-interior-approximation}}
\hfill
\subfloat[]{%
\includegraphics[width=0.485\linewidth]{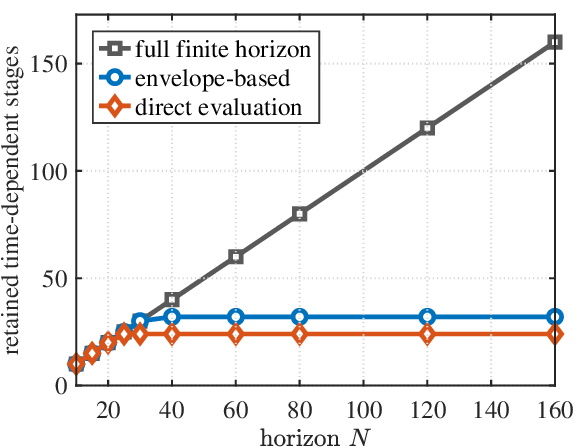}%
\label{fig:numerical-interior-stage-advantage}}
\caption{Interior approximation for
\(\varepsilon=10^{-4}\). (a) Interior error \(E_{N,L}\) and its (a posteriori)
numerical envelope. The envelope gives \(L=16\), while direct evaluation gives
\(E_{80,12}\le\varepsilon\). (b) Numbers of retained time-dependent stages
for the full finite-horizon sequences, the envelope-based approximation, and
direct evaluation.}
\label{fig:numerical-interior-results}
\end{figure}

The marked value \(L=16\) in
Figure~\ref{fig:numerical-interior-approximation} is obtained from the a
posteriori envelope. It is not the theoretically defined boundary-layer width
\(L_\varepsilon^{\rm int}\) unless explicit values of \(C_{\rm int}\) and
\(\rho_{\rm int}\) are used. Figure~\ref{fig:numerical-interior-stage-advantage}
reports how many stages retain the finite-horizon mean variables and
covariance. It does not measure the total cost of solving the finite-horizon
minimax problem.

The finite-horizon recursions are evaluated at every stage in
Figures~\ref{fig:numerical-mean-turnpike}--\ref{fig:numerical-interior-results}.
Errors below the numerical display threshold are omitted from the
semilogarithmic curves rather than replaced by a positive lower bound.

\subsection{Policy comparison and coefficient computation time}

Finally we verify the approximation error of the value function when using the hybrid quasi-turnpike policy in Section \ref{subsec:quasi-turnpike}. We measure the relative value gap by
$g_q:=\frac{J_N^{\rm wc}(\pi^{(q)})-J_N^*}{|J_N^*|}$.

\begin{figure}[!t]
\centering
\subfloat[]{%
\includegraphics[width=0.485\linewidth]{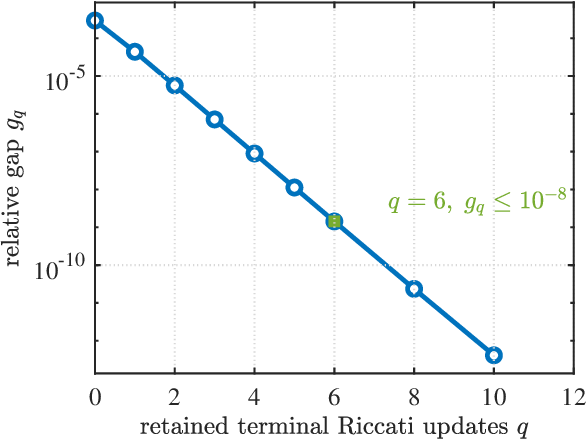}%
\label{fig:numerical-policy-value-gap}}
\hfill
\subfloat[]{%
\includegraphics[width=0.485\linewidth]{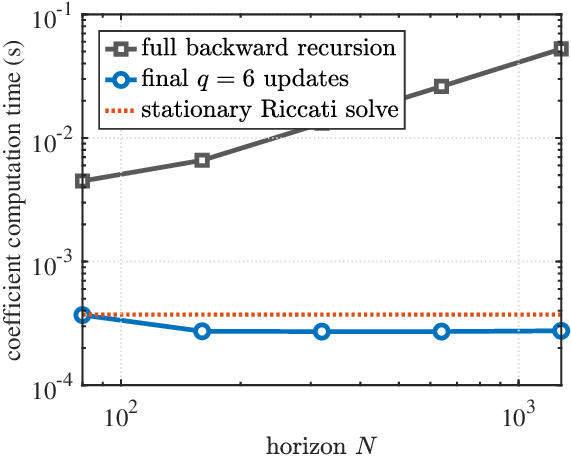}%
\label{fig:numerical-policy-runtime}}
\caption{ Policy comparison and
coefficient computation time. (a) Relative value gap versus the
number \(q\) of retained terminal stages. (b) Computation time for the full
recursion, the final six updates, and the one-time computation of
\(P_{\rm ss}\).}
\label{fig:numerical-policy-results}
\end{figure}


All reported relative value gaps are evaluated under the original uncentered
empirical distribution. The worst-case disturbance distribution is recomputed
for each policy. In this example, retaining six terminal stages gives a
relative value gap below \(10^{-8}\). The policy designed with \(d=0\) has a
larger relative value gap. The relative value gaps here are numerical quantities and are not themselves bounded by Proposition~\ref{prp:quasi-turnpike-cost}, which instead provides a horizon-uniform exponential bound on the absolute worst-case cost gap.

Figure~\ref{fig:numerical-policy-runtime} reports the median of 15
measurements after two warm-up runs. Each measurement is repeated until the
accumulated time exceeds \(0.15\) seconds. These measurements include only the
computation of the Riccati matrices and the policy coefficients \(K_t\) and
\(k_t\). They exclude closed-loop simulation, worst-case value evaluation,
and the total runtime of the complete dynamic program.





\subsection{Sensitivity to the penalty parameter and a 6-D example}

We test $\lambda\in\{0.10,0.12,0.16,0.20,0.22,0.23,0.24,0.25$, $0.27,0.30,0.40,0.50,1,2,4,8\}$.
For each value satisfying the limiting penalty condition, we compute
\(P_{\rm ss}\), the smallest eigenvalues of \(W\) and
\(\lambda I-\Xi^\top P_{\rm ss}\Xi\), and the reference covariance. Numerical
bisection gives the value at which
\(\lambda_{\min}(\lambda I-\Xi^\top P_{\rm ss}\Xi)\) crosses zero. We denote
this estimate by \(\widehat\lambda_{\rm th}\).

\begin{figure}[!t]
\centering
\subfloat[]{%
\includegraphics[width=0.485\linewidth]{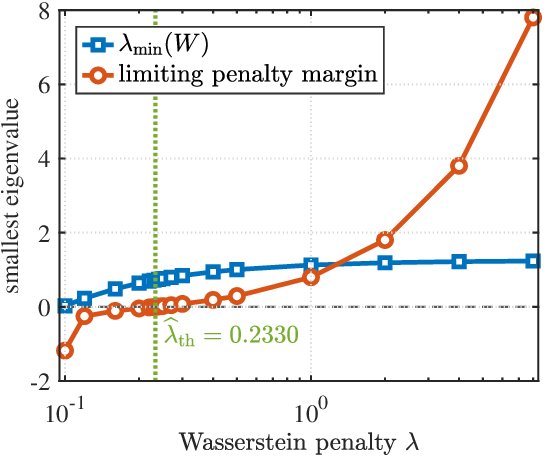}%
\label{fig:numerical-lambda-margins}}
\hfill
\subfloat[]{%
\includegraphics[width=0.485\linewidth]{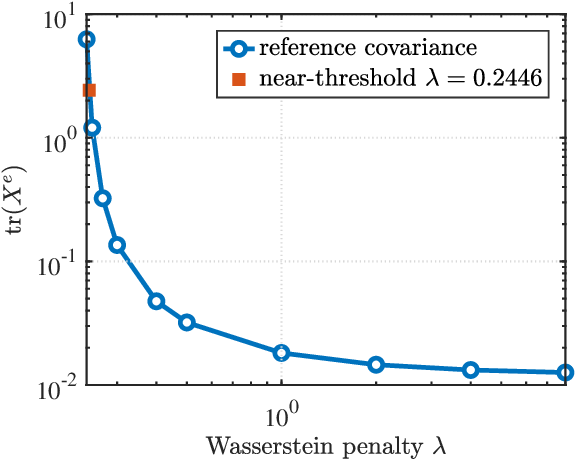}%
\label{fig:numerical-lambda-covariance}}
\caption{Sensitivity to the Wasserstein penalty. (a) Smallest
eigenvalues of \(W\) and \(\lambda I-\Xi^\top P_{\rm ss}\Xi\). The vertical
line marks \(\widehat\lambda_{\rm th}\approx0.2330\). (b)
\(\operatorname{tr}(X^e)\). The marked point corresponds to
\(\lambda=0.2446\).}
\label{fig:numerical-lambda-sensitivity}
\end{figure}

Figure~\ref{fig:numerical-lambda-margins} shows that \(W\succeq0\) does not
imply the limiting penalty condition. Figure~\ref{fig:numerical-lambda-covariance}
shows a sharp increase in \(\operatorname{tr}(X^e)\) near the estimated
threshold. This is a numerical sensitivity result; no monotonicity in \(\lambda\) is asserted.

For a second example, we use a six-dimensional state, three control inputs, three disturbance components, and \(N=240\). The matrix \(B\in\mathbb R^{6\times3}\) has rank \(3\), and \(W\succeq0\) is singular
with rank \(3\). The spectral radius of \(A_c\) is approximately \(0.92\).
The ratio of the largest to the smallest eigenvalue of the empirical
covariance is \(80\). The stabilizability, observability, and limiting penalty
conditions are satisfied. The system matrices are
\begin{equation*}
\setlength{\arraycolsep}{1.8pt}
\begin{gathered}
A=
\!\left[\begin{smallmatrix}
1.03 & 0.08 & 0    & 0.02 & 0    & 0\\
0    & 1.01 & 0.06 & 0    & 0.02 & 0\\
0    & 0    & 0.99 & 0    & 0    & 0.02\\
0    & 0    & 0    & 0.92 & 0.05 & 0\\
0    & 0    & 0    & 0    & 0.88 & 0.04\\
0    & 0    & 0    & 0    & 0    & 0.84
\end{smallmatrix}\right],
~B=\!
\left[\begin{matrix}
\diag(1,0.8,0.65)\\[1pt]
0_{3\times3}
\end{matrix}\right],\\
\Xi=B\diag(0.36,0.30,0.25),\\
Q=\diag(1,0.9,0.8,0.4,0.3,0.2),\\
R=\diag(1,1.2,0.9),\qquad
Q_f=Q,\qquad \lambda=0.30.
\end{gathered}
\end{equation*}
The initial and empirical moments are
\begin{equation*}
\setlength{\arraycolsep}{2pt}
\begin{gathered}
m_0=(1.5,-1.2,0.8,0.5,-0.4,0.3)^\top,\\
X_0=\diag(0.30,0.22,0.18,0.10,0.08,0.06),\\
d=(0.18,-0.08,0.05)^\top,
~\Sigma=
\left[\begin{smallmatrix}
 0.06525 &  0.05475 &  0.0045\\
 0.05475 &  0.06525 & -0.0045\\
 0.0045  & -0.0045  &  0.006
\end{smallmatrix}\right].
\end{gathered}
\end{equation*}
The empirical distribution has \(M=18\) equally weighted support points with
these moments. For this linear-quadratic model, the reported mean and
covariance recursions depend on the empirical support only through \(d\) and
\(\Sigma\).

\begin{figure}[!t]
\centering
\subfloat[]{%
\includegraphics[width=0.485\linewidth]{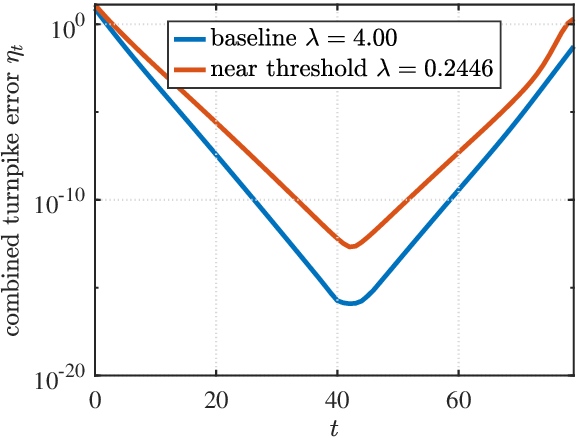}%
\label{fig:numerical-near-threshold}}
\hfill
\subfloat[]{%
\includegraphics[width=0.485\linewidth]{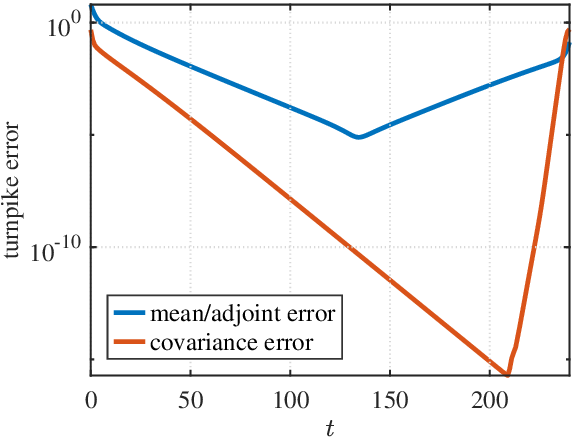}%
\label{fig:numerical-high-dimensional}}
\caption{Near-threshold and six-state examples. (a) Combined
error \(\eta_t\) for \(\lambda=4\) and \(\lambda=0.2446\). (b)
Mean-state/adjoint and covariance errors in the six-state example.}
\label{fig:numerical-challenging-cases}
\end{figure}

Figure~\ref{fig:numerical-near-threshold} retains the two-sided turnpike form
for both penalty values, while the near-threshold case has larger errors.
Values near the center can reach numerical precision and are not used for
quantitative comparison. In Figure~\ref{fig:numerical-high-dimensional}, the
blue curve is \(\|m_t^*-m^e\|+\|p_t^*-p^e\|\), and the orange curve is
\(\|X_t^*-X^e\|_F\). Both errors retain the two-sided turnpike form. The
mean-state and mean-adjoint error decays more slowly than in the
two-dimensional example.



\section{Conclusion}\label{sec:conclusion}


In this paper, we study the Wasserstein-penalized linear-quadratic control problem over a long horizon and establish the turnpike estimates, showing that for a prescribed tolerance, the finite-horizon mean variables and covariance can be approximated by their time-independent references outside the initial and terminal boundary layers. 

As a control consequence, we constructed a hybrid quasi-turnpike policy that uses time-independent affine feedback away from the terminal boundary and finite-horizon optimal coefficients over the final $q$ stages whose worst-case finite-horizon cost gap decays exponentially in $q$, with constants independent of the horizon. Thus, for a prescribed performance accuracy, the number of time-dependent terminal coefficients that must be retained need not grow with the horizon, and the fraction of the horizon over which time-independent approximations can be used approaches one as the horizon increases. 

Several questions remain open. First, under what additional assumptions can the mean and covariance estimates be strengthened to a 2-Wasserstein turnpike estimate for the actual closed-loop state distributions, rather than only for their Gaussian moment surrogates? Second, when the empirical mean and covariance vary with time, can suitable time-dependent references be characterized through nonautonomous Hamiltonian and Lyapunov equations while retaining estimates uniform in the horizon? Extensions to infinite-dimensional and nonlinear systems, as well as applications of the turnpike structure to model predictive control and reinforcement-learning-based control schemes \cite{Breiten2020,Reiter2025}, also remain of interest.






\appendices

\section{Proofs in Section \ref{sec:bellman-riccati}}

\subsection{Proof of Lemma~\ref{lem:one-step-affine-minimax}}
\label{app:one-step-affine-minimax}

\begin{proof}
	Substituting \eqref{eq:one-step-continuation} into the maximization over \(w_t\) in \eqref{eq:Bellman-finite-support}, the objective function associated with the empirical sample \(\widehat w_t^{(i)}\) has Hessian
	\(
	2(\Xi^\top P_{t+1}\Xi-\lambda I)
	\)
	with respect to \(w_t\). By \eqref{eq:one-step-admissibility}, this Hessian is negative definite, hence the maximizer is unique and
    satisfies
	\begin{equation}
		\begin{aligned}
			0\!=\!
			\Xi^\top\! P_{t+1}(Ax_t\!+\!\!Bu_t\!+\!\Xi w_t)
			+\Xi^\top \!s_{t+1} \!-\!\lambda(w_t\!-\!\widehat w_t^{(i)}),
		\end{aligned}
		\label{eq:w-foc-revised}
	\end{equation}
	which gives
	\begin{equation}
		\begin{aligned}
			w_t^{*,i}(x_t,u_t)
			={}&
			(\lambda I-\Xi^\top P_{t+1}\Xi)^{-1}        \\
			&\!\!\!\!\!\!\!\!\!\!\!\!\!\!\!\!\!\!\!\times
			\Big[
			\Xi^\top P_{t+1}(Ax_t+Bu_t)
			+\Xi^\top s_{t+1}
			+\lambda\widehat w_t^{(i)}
			\Big].
		\end{aligned}
		\label{eq:wstar-affine-xu}
	\end{equation}
	
	After substituting \eqref{eq:wstar-affine-xu} into the right-hand side of \eqref{eq:Bellman-finite-support}, the objective function in \(u_t\) is strictly convex. Its Hessian is
	\(
	2\Big(
	R+B^\top P_{t+1}B
	+B^\top P_{t+1}\Xi
	(\lambda I-\Xi^\top P_{t+1}\Xi)^{-1}
	\Xi^\top P_{t+1}B
	\Big)\succ0,
	\)
	because \(R\succ0\) and \(\lambda I-\Xi^\top P_{t+1}\Xi\succ0\). Therefore, the minimizing input is unique. Differentiating the objective function with respect to \(u_t\), and using \eqref{eq:w-foc-revised} yields 
	\begin{equation}
		\begin{aligned}
			0={}&
			Ru_t^*
			+B^\top
			\Bigg[
			P_{t+1}(Ax_t+Bu_t^*)
			+s_{t+1}  \\
			&\qquad\qquad
			+P_{t+1}\Xi
			\frac1M\sum_{i=1}^M
			w_t^{*,i}(x_t,u_t^*)
			\Bigg].
		\end{aligned}
		\label{eq:u-foc-revised}
	\end{equation}
	
	Averaging \eqref{eq:wstar-affine-xu} over \(i=1,\ldots,M\) and using the definition of \(d_t\) gives
	\begin{equation}
		\begin{aligned}
			\frac1M\sum_{i=1}^M
			w_t^{*,i}(x_t,u_t^*)  
			& =
			d_t
			+\lambda^{-1}\Xi^\top
			\Bigg[
			P_{t+1}(Ax_t+Bu_t^*)  \\
			&\!\!\!\!\!\!\!\!\!+s_{t+1}
			+P_{t+1}\Xi
			\frac1M\sum_{i=1}^M
			w_t^{*,i}(x_t,u_t^*)
			\Bigg].
		\end{aligned}
		\label{eq:wavg-revised}
	\end{equation}
	
	Combining \eqref{eq:u-foc-revised} and \eqref{eq:wavg-revised} yields
	\begin{equation}
		\begin{aligned}
			&(I+P_{t+1}W)
			\Bigg[
			P_{t+1}(Ax_t+Bu_t^*)
			+s_{t+1}  \\
			&\qquad\qquad
			+P_{t+1}\Xi
			\frac1M\sum_{i=1}^M
			w_t^{*,i}(x_t,u_t^*)
			\Bigg]  \\
			&\quad =
			P_{t+1}Ax_t+P_{t+1}\Xi d_t+s_{t+1}.
		\end{aligned}
		\label{eq:combined-gradient-equation}
	\end{equation}
	The matrix \(I+P_{t+1}W\) is nonsingular under
	\eqref{eq:one-step-admissibility}. To see this, note that
	\(
	I+P_{t+1}^{1/2}WP_{t+1}^{1/2}
	=
	I-\lambda^{-1}P_{t+1}^{1/2}\Xi\Xi^\top P_{t+1}^{1/2}
	+P_{t+1}^{1/2}BR^{-1}B^\top P_{t+1}^{1/2}
	\succ0,
	\)
	and
	\(
	\det(I+P_{t+1}W)
	=
	\det(I+P_{t+1}^{1/2}WP_{t+1}^{1/2}).
	\)
	Hence \eqref{eq:combined-gradient-equation} gives
	\begin{equation*}
		\begin{aligned}
			& P_{t+1}(Ax_t+Bu_t^*)
			+s_{t+1}  
			+P_{t+1}\Xi
			\frac1M\sum_{i=1}^M
			w_t^{*,i}(x_t,u_t^*)  \\
			&=
			(I+P_{t+1}W)^{-1}
			\big(P_{t+1}Ax_t+P_{t+1}\Xi d_t+s_{t+1}\big).
		\end{aligned}
	\end{equation*}
	Substituting it into
	\eqref{eq:u-foc-revised} yields
	\begin{equation*}
			u_t^*(x)\!
	=\!
	-R^{-1}\!B^\top\!
	(I+P_{t+1}W)^{-1} \!
	\big(
	P_{t+1}Ax
	+\!P_{t+1}\Xi d_t
	+s_{t+1}
	\big)
	\end{equation*}
	which proves \eqref{eq:affine-control}--\eqref{eq:k-affine}.
	
	Substituting \(u_t^*=K_tx_t+k_t\) into
	\eqref{eq:wstar-affine-xu} gives
	\eqref{eq:wstar-affine-state}--\eqref{eq:G-affine}.
\end{proof}

\subsection{Proof of Theorem~\ref{thm:affine-bellman-reduction}}
\label{app:thm1}

\begin{proof} Assumption~\ref{ass:riccati-penalty} gives the one-step condition \eqref{eq:one-step-admissibility} for each backward step, so Lemma~\ref{lem:one-step-affine-minimax} can be applied recursively.
	The proof is by backward induction. At \(t=N\), the claim follows from
	\(V_N(x_N)=x_N^\top Q_fx_N\). Suppose that
	\eqref{eq:affine-value} holds at stage \(t+1\).
	Lemma~\ref{lem:one-step-affine-minimax} gives the unique solution of the minimization problem in \(u_t\) and the corresponding support points of a worst-case disturbance distribution.
	
	Substituting the minimizer \(u_t^{*}\) and the support points
\(w_t^{*,i}\), \(i=1,\ldots,M\), into \eqref{eq:Bellman-finite-support}
	shows that \(V_t\) has the form \eqref{eq:affine-value}. Differentiating
	the right-hand side of \eqref{eq:Bellman-finite-support} with respect to
	\(x_t\), and using the first-order optimality conditions in the
	one-step problem, gives
	\begin{align*}
			P_tx_t+s_t
			={}&
			Qx_t
			+A^\top
			(I+P_{t+1}W)^{-1} \\
			&\quad\times
			\big(
			P_{t+1}A x_t
			+P_{t+1}\Xi d_t
			+s_{t+1}
			\big).
	\end{align*}
	Matching the coefficient of \(x_t\) and the constant term in the above equation gives
	\eqref{eq:P-recursion-affine} and \eqref{eq:s-recursion-affine}.
	
	It remains to identify the scalar term. Evaluating the right-hand side
	of \eqref{eq:Bellman-finite-support} at \(x_t=0\), and using
	\(u_t^*=k_t\) and \(w_t^{*,i}=G_t^i\), gives
	\eqref{eq:z-recursion-affine}--\eqref{eq:Z-increment}. This completes
	the induction step. 
\end{proof}

\subsection{Proof of Proposition~\ref{prop:riccati-convergence}}
\label{app:pro1}

\begin{proof}
	By Assumption~\ref{ass:psd-W}, for every \(P\succeq0\),
	\(
		I+W^{1/2}PW^{1/2}\succ0 .
	\)
	By Sylvester's determinant theorem,
	\(
		\det(I+PW)
		=
		\det(I+W^{1/2}PW^{1/2})>0 .
	\)
	Thus \(I+PW\) is nonsingular for every \(P\succeq0\), and
	\(\mathcal R\) is well defined on the positive-semidefinite cone.
	
	The matrix inversion lemma gives
	\begin{equation*}
			(I+PW)^{-1}P
			=
			P
			-
			PW^{1/2}
			(I+W^{1/2}PW^{1/2})^{-1}
			W^{1/2}P .
	\end{equation*}
	Using it in the definition of
	\(\mathcal R\), we obtain
	\begin{equation}
		\begin{aligned}
			\mathcal R(P)
			=&
			~Q+A^\top PA\\
			&-\!
			A^\top\! PW^{1/2}
			(I+W^{1/2}PW^{1/2})^{-1}
			W^{1/2}PA .
		\end{aligned}
		\label{eq:standard-riccati-form}
	\end{equation}
	Hence \(\mathcal R\) coincides with the standard discrete-time Riccati
	operator associated with the auxiliary input matrix \(W^{1/2}\), state
	weight \(Q\), and input weight \(I\).
	
    By Assumption~\ref{ass:psd-W}, the standard discrete-time Riccati convergence theorem applies to \eqref{eq:standard-riccati-form}. Therefore, starting from \(S_0=Q_f\succeq0\), the sequence \eqref{eq:S-orbit} remains positive semidefinite and bounded, and converges to the stabilizing positive-semidefinite solution \(P_{\rm ss}\) of \(P=\mathcal R(P)\).
	
	The stabilizing closed-loop matrix of the auxiliary Riccati equation is $A-W^{1/2}(I+W^{1/2}P_{\rm ss}W^{1/2})^{-1}W^{1/2}P_{\rm ss}A$.
	Applying the matrix inversion again gives
	$$
	(I+WP_{\rm ss})^{-1}=I-W^{1/2}(I+W^{1/2}P_{\rm ss}W^{1/2})^{-1}W^{1/2}P_{\rm ss}.
        $$
	Thus the stabilizing closed-loop matrix is equal to
	\((I+WP_{\rm ss})^{-1}A=A_c\). Since \(P_{\rm ss}\) is  stabilizing, \(A_c\) is Schur.
\end{proof}

\subsection{Proof of Lemma~\ref{lem:closed-tail-margin}}
\label{app:lem2}

\begin{proof}
	By Proposition~\ref{prop:riccati-convergence}, \(S_j\to P_{\rm ss}\).
	Thus \(\mathscr S\) is compact because it consists of a convergent
	sequence and its limit.
	
	For a symmetric matrix \(Y\), let \(\lambda_{\min}(Y)\) denote its smallest eigenvalue. For each \(S_j\),  Assumption~\ref{ass:riccati-penalty} gives \(\lambda I-\Xi^\top S_j\Xi\succ0\), and the same inequality holds at \(P_{\rm ss}\). Hence
	\(
	S\mapsto
	\lambda_{\min}(\lambda I-\Xi^\top S\Xi)
	\)
	is positive on \(\mathscr S\). Since this map is continuous,  it has a positive minimum on the compact set \(\mathscr S\).
	
	It remains to bound the matrices \(I+SW\) and \(I+WS\). For every
	\(S\in\mathscr S\), we have \(S\succeq0\) and \(W\succeq0\). The
	nonzero eigenvalues of \(SW\) and \(WS\) coincide with the nonzero
	eigenvalues of \(W^{1/2}SW^{1/2}\), which is positive semidefinite.
	Hence \(SW\) and \(WS\) have no eigenvalue equal to \(-1\), and
	therefore \(I+SW\) and \(I+WS\) are nonsingular for every
	\(S\in\mathscr S\).
	
	The maps $S\mapsto \sigma_{\min}(I+SW)$,
	$S\mapsto \sigma_{\min}(I+WS)$
	are continuous and positive on the compact set \(\mathscr S\). Their
	minima are positive. Taking \(\delta\) as the minimum of the three
	positive lower bounds proves \eqref{eq:uniform-admissibility-orbit}.
\end{proof}

\section{Proofs in Section \ref{sec:turnpike}}

\subsection{Proof of Lemma~\ref{lem:uniform-product}}
\label{app:lem4}

\begin{proof}
	The stated assumptions allow us to use Lemmas~\ref{lem:closed-tail-margin} and \ref{lem:riccati-tail-exp}.
	Set
	\(
	F_\ell:=A_c(P_\ell)=A_c(S_{N-\ell}).
	\)
	If \(N-\ell\ge J_L\), then \(F_\ell\) satisfies
	\(
	(F_\ell)^\top H_LF_\ell-H_L
	\preceq
	-\frac12 I .
	\)
	Therefore, with \(\|x\|_{H_L}:=(x^\top H_Lx)^{1/2}\),
	\[
	\begin{aligned}
		\|F_\ell x\|_{H_L}^2
		\le
		\|x\|_{H_L}^2-\frac12\|x\|^2 \le
		\left(1-\frac{1}{2\lambda_{\max}(H_L)}\right)
		\|x\|_{H_L}^2 .
	\end{aligned}
	\]
	Thus each factor with \(N-\ell\ge J_L\) has induced \(H_L\)-norm at most
	some \(\alpha\in(0,1)\).
	
	The remaining factors satisfy \(N-\ell<J_L\). There are at most \(J_L\)
	such factors in any product, and their \(H_L\)-induced norms are uniformly
	bounded because \(\mathscr S\) is compact. Let
	\(
	M_L:=\sup_{S\in\mathscr S}\|A_c(S)\|_{H_L\to H_L}<\infty .
	\)
	For a product of length \(k=t-s\), $\|\Phi_N(t,s)\|_{H_L\to H_L}
		\le
		M_L^{\min\{k,J_L\}}\alpha^{(k-J_L)_+}$.
	Choose any \(\rho\in(\alpha,1)\). Increasing the multiplicative
	constant if necessary, we have
	$\|\Phi_N(t,s)\|_{H_L\to H_L}
	\le
	C_L\rho^k$.
	Equivalence of the \(H_L\)-norm and the Euclidean norm gives
	\eqref{eq:uniform-product-bound}. 
\end{proof}

\subsection{Proof of Theorem~\ref{thm:affine-mean-turnpike}}
\label{app:thm2}

\begin{proof}
	The assumptions give the error recursions
	\eqref{eq:shifted-ham-bvp} and the product estimate in
	Lemma~\ref{lem:uniform-product}.
	Iterating \eqref{eq:r-backward} from \(t\) to \(N\), and using
	\(r_N=r^e\), gives
	$r_t
		=
		(A_c(P_{t+1}))^\top
		(A_c(P_{t+2}))^\top
		\cdots      
		(A_c(P_N))^\top r^e$.
	This product is the transpose of \(\Phi_N(N,t)\). Hence
	Lemma~\ref{lem:uniform-product} gives
	$\|r_t\|
		\le
		C_\Phi\rho^{N-t}\|r^e\|$, $0\le t\le N$.
        
Using \eqref{eq:e-driven1} and the definition of \(\Phi_N\),
	\begin{equation*}
			e_t
			=
			\Phi_N(t,0)e_0 -
			\sum_{\ell=0}^{t-1}
			\Phi_N(t,\ell+1)
			(I+WP_{\ell+1})^{-1}
			Wr_{\ell+1} .
	\end{equation*}
	Let
	$M_W:=\underset{P\in\mathscr S}{\sup}\|(I+WP)^{-1}W\|$,
	which is finite by Lemma~\ref{lem:closed-tail-margin}. Combining the above estimates and
	Lemma~\ref{lem:uniform-product}, we obtain
	\begin{align*}
		\|e_t\|
		\le
		C_\Phi\rho^t\|e_0\|
		\notag+
		C_\Phi^2M_W
		\sum_{\ell=0}^{t-1}
		\rho^{t-\ell-1}
		\rho^{N-\ell-1}
		\|r^e\|.
	\end{align*}
	The sum on the right-hand side satisfies
	\[
	\begin{aligned}
		\sum_{\ell=0}^{t-1}
		\rho^{t-\ell-1}
		\rho^{N-\ell-1}
		=
		\rho^{N-t}
		\sum_{j=0}^{t-1}\rho^{2j}\le
		\frac{1}{1-\rho^2}
		\rho^{N-t}.
	\end{aligned}
	\]
	Therefore, there is a constant \(C_e>0\), independent of \(N,t\), such
	that $\|e_t\|\le C_e\Big(\rho^t\|m_0-m^e\|+\rho^{N-t}\|r^e\|\Big)$.
	
	Since \(q_t=P_te_t+r_t\) and
	\(P_t\in\mathscr S\), compactness of \(\mathscr S\) gives
	\(
	\sup_{P\in\mathscr S}\|P\|<\infty .
	\)
	Combining this bound with the bounds on $\|r_t\|$ and $\|e_t\|$
	yields a constant \(C_q>0\) such that
	\begin{equation}
		\|q_t\|
		\le
		C_q
		\Big(
		\rho^t\|m_0-m^e\|
		+
		\rho^{N-t}\|r^e\|
		\Big).
		\label{eq:q-bound}
	\end{equation}
	Adding the bounds on $\|e_t\|$ and $\|q_t\|$ proves \eqref{eq:two-sided-mp} after increasing the
	constant.
	
	It remains to estimate the mean control input and the mean of the
worst-case disturbance distribution.
	From \eqref{eq:ue-we} and \eqref{eq:mean-u-w-affine},
	\(
	\bar{u}_t^*-u^e
	=
	-R^{-1}B^\top q_{t+1},
	\bar{w}_t^*-\bar w^e
	=
	\lambda^{-1}\Xi^\top q_{t+1} .
	\)
	Applying \eqref{eq:q-bound} at stage \(t+1\) gives
	$$
		\|q_{t+1}\|
		\le
		C
		\Big(
		\rho^t\|m_0-m^e\|
		+
		\rho^{N-t-1}\|r^e\|
		\Big).
    $$
	Multiplying by  \(R^{-1}B^\top\) and
	\(\lambda^{-1}\Xi^\top\) proves \eqref{eq:two-sided-uw}.
\end{proof}

\subsection{Proof of Theorem~\ref{thm:cov-turnpike}}
\label{app:thm3}

\begin{proof}
	First, the maps \(P\mapsto A_c(P)\) and \(P\mapsto D(P)\) are Lipschitz on
	\(\mathscr S\). Indeed, for \(P,\widetilde P\in\mathscr S\),
	\[
	\begin{aligned}
		&(I+WP)^{-1}-(I+W\widetilde P)^{-1}       \\
		=&-(I+WP)^{-1}
		W(P-\widetilde P)
		(I+W\widetilde P)^{-1},\\
		& (\lambda I-\Xi^\top P\Xi)^{-1}
		-(\lambda I-\Xi^\top \widetilde P\Xi)^{-1} \\
		=&
		(\lambda I-\Xi^\top P\Xi)^{-1}
		\Xi^\top(P-\widetilde P)\Xi
		(\lambda I-\Xi^\top \widetilde P\Xi)^{-1}.
	\end{aligned}
	\]
	By Lemma \ref{lem:closed-tail-margin}, the inverse matrices are bounded uniformly
	on \(\mathscr S\). Hence there is \(L>0\) such that, for all
	\(P,\widetilde P\in\mathscr S\),
	$$
			\|A_c(P)-A_c(\widetilde P)\|
			+\|D(P)-D(\widetilde P)\| \le
			L\|P-\widetilde P\|.
	$$
	Since \(P_{t+1}=S_{N-t-1}\), \eqref{eq:riccati-exp-tail} gives
	\begin{equation}
		\begin{aligned}
			\|A_c(P_{t+1})\!-\!A_c\|
			+\|D(P_{t+1})\!-\!D(P_{\rm ss})\| \le
			C\rho_P^{N-t-1}
		\end{aligned}
		\label{eq:FD-tail-estimate}
	\end{equation}
	for all $0\le t\le N-1$.
	Let \(\Delta_t:=X_t^*-X^e\). Subtracting
	\eqref{eq:Xe-Lyapunov} from \eqref{eq:cov-recursion} gives
	\begin{equation}
		\Delta_{t+1}
		=
		A_c(P_{t+1})
		\Delta_t
		A_c(P_{t+1})^\top
		+
		\mathcal E_t,
		\label{eq:Delta-cov-recursion}
	\end{equation}
	where
	\(
		\mathcal E_t
		:=
		A_c(P_{t+1})X^eA_c(P_{t+1})^\top
		-
		A_cX^eA_c^\top  +
		D(P_{t+1})-D(P_{\rm ss}) .
	\)
	Because \(A_c(P_{t+1})\) is uniformly bounded and \(X^e\) is fixed,
	\eqref{eq:FD-tail-estimate} implies $\|\mathcal E_t\|\le C\rho_P^{N-t-1}$.
	Iterating \eqref{eq:Delta-cov-recursion} gives
	\begin{equation*}
		\begin{aligned}
			\Delta_t
		\!=\!
			\Phi_N(t,0)
			\Delta_0
			\Phi_N(t,0)^\top\!\!\! +\!\!
			\sum_{\ell=0}^{t-1}\!
			\Phi_N(t,\ell\!+\!1)
			\mathcal E_\ell
			\Phi_N(t,\ell\!+\!1)^\top .
		\end{aligned}
	\end{equation*}
	Using the estimate of $\|\mathcal{E}_t\|$ and  Lemma~\ref{lem:uniform-product},
	\begin{equation}
		\begin{aligned}
			\|\Delta_t\|
			\le
			C\rho^{2t}\|\Delta_0\|+
			C
			\sum_{\ell=0}^{t-1}
			\rho^{2(t-\ell-1)}
			\rho_P^{N-\ell-1}.
		\end{aligned}
		\label{eq:Delta-cov-bound-before-sum}
	\end{equation}
	where
		$\sum_{\ell=0}^{t-1}
		\rho^{2(t-\ell-1)}
		\rho_P^{N-\ell-1}
		=\rho_P^{N-t}
		\sum_{j=0}^{t-1}
		(\rho^2\rho_P)^j \le C\rho_P^{N-t}.$
        
	Choose any \(\rho_X\in(\max\{\rho^2,\rho_P\},1)\). Then
	\(\rho^{2t}\le\rho_X^t\) and
	\(\rho_P^{N-t}\le\rho_X^{N-t}\). Hence
	$\|\Delta_t\|\le C_X\left(\rho_X^t\|\Delta_0\|+\rho_X^{N-t}\right)$.
	Since \(\Delta_0=X_0-X^e\), this proves
	\eqref{eq:cov-turnpike-bound}.
\end{proof}

\subsection{Proof of Corollary~\ref{cor:gaussian-surrogate-turnpike}}
\label{app:cro1}

\begin{proof}
Let \(Z\sim\mathcal N(0,I_n)\) be defined on an auxiliary
			probability space. Then
	\(
	m_t^*
	+
	\big(X_t^*\big)^{1/2}Z
	\sim
	\mathcal G_t,
	m^e+(X^e)^{1/2}Z
	\sim
	\mathcal G^e .
	\)
    These two random vectors form an admissible coupling of \(\mathcal G_t\) and \(\mathcal G^e\). By the definition of \(W_2\),
		\begin{align}
			W_2\big(\mathcal G_t,\mathcal G^e\big)
			&\le
			\Bigg(
			\|m_t^*-m^e\|^2 \!+\!
			\Big\|
			\big(X_t^*\big)^{1/2}
			\!\!-\!
			(X^e)^{1/2}
			\Big\|_F^2
			\Bigg)^{1/2}                                     \nonumber \\
			&\le
			\|m_t^*-m^e\| +
			\Big\|
			\big(X_t^*\big)^{1/2}
			-
			(X^e)^{1/2}
			\Big\|_F .\label{eq:gaussian-coupling-bound}
		\end{align}

	For positive semidefinite matrices, the
			principal square-root map satisfies the global \(1/2\)-H\"older
			estimate
	\begin{equation*}
		\|X^{1/2}-Y^{1/2}\|_F
		\le
		C_{\rm sq}\|X-Y\|^{1/2},
		\qquad
		X,Y\succeq0,
	\end{equation*}
	where \(C_{\rm sq}>0\) depends only on the dimension and on the matrix norm used in \eqref{eq:cov-turnpike-bound}. Applying this estimate to \eqref{eq:gaussian-coupling-bound}
	gives $W_2\big(\mathcal G_t,\mathcal G^e\big)
		\le
		\|m_t^*-m^e\| +
		C_{\rm sq}
		\|X_t^*-X^e\|^{1/2}$.
	
	Using \eqref{eq:two-sided-mp} and \eqref{eq:cov-turnpike-bound}, and then applying \(\sqrt{a+b}\le\sqrt a+\sqrt b\), we obtain
	\[
	\begin{aligned}
		W_2\big(\mathcal G_t,\mathcal G^e\big)
		\le&
		C_m\rho^t\|m_0-m^e\|+
		C_m\rho^{N-t}\|Q_fm^e-p^e\| \\
		&\!\!\!\!\!\!\!\!\!+
		C_{\rm sq}C_X^{1/2}
		\rho_X^{t/2}
		\|X_0\!-\!X^e\|^{1/2} \!+\!
		C_{\rm sq}C_X^{1/2}
		\rho_X^{(N-t)/2}.
	\end{aligned}
	\]
	Taking
	\(
	\rho_G
	:=
	\max\{\rho,\sqrt{\rho_X}\}
	\in(0,1),
	\)
	we have
	$\rho^t\le\rho_G^t$,
	$\rho_X^{t/2}\le\rho_G^t$,
	$\rho^{N-t}\le\rho_G^{N-t}$,
	$\rho_X^{(N-t)/2}\le\rho_G^{N-t}$.
	Since
			\(\|Q_fm^e-p^e\|\) is fixed, increasing \(C_G\) to absorb this
			quantity and the fixed constants \(C_m\), \(C_X\), and
			\(C_{\rm sq}\) proves \eqref{eq:gaussian-surrogate-W2-bound}.
\end{proof}

\subsection{Proof of Proposition \ref{prp:quasi-turnpike-cost}}\label{app:prop4}

\begin{proof}
	Set $s_{\rm ss}:=p^e-P_{\rm ss}m^e$. We first note that there exist
	$C_0>0$ and $\rho_0\in(0,1)$ such that
	\begin{align}\label{72}
	    \|K_t-K_{\rm ss}\|+\|k_t-k_{\rm ss}\|
	\le C_0\rho_0^{\,N-t-1}.
	\end{align}
	Indeed, Lemma \ref{lem:riccati-tail-exp} gives the corresponding estimate for
	$P_{t+1}-P_{\rm ss}$. To estimate $s_{t+1}$, apply Theorem \ref{thm:affine-mean-turnpike} to the
	initial mean $m_0=m^e$. Since
	$p_t^\ast=P_tm_t^\ast+s_t$, we obtain
	\[
	\begin{aligned}
		\|s_t\!-\!s_{\rm ss}\|
		&\le
		\|p_t^\ast\!-\!p^e\|
		\!+\!\|P_t\|\,\|m_t^\ast\!-\!m^e\|
		\!+\!\|P_t\!-\!P_{\rm ss}\|\,\|m^e\|  \\
		&\le C\rho_0^{\,N-t},
	\end{aligned}
	\]
	after increasing $C$ and $\rho_0$ if necessary. Equation \eqref{72} then
	follows from \eqref{eq:K-affine}-\eqref{eq:k-affine}, Lemma \ref{lem:closed-tail-margin}, and the definitions of
	$K_{\rm ss}$ and $k_{\rm ss}$.
	
	For the optimal continuation value $V_{t+1}$, define
	\begin{align*}
	    \mathcal{Q}_t(x,u):=&~
	x^\top Qx+u^\top Ru\\
    &\!\!\!\!\!\!\!\!\!\!\!\!\!\!\!\!\!\!\!\!\!\!+
	\frac1M\sum_{i=1}^M
	\sup_w
	\left\{
	V_{t+1}(Ax+Bu+\Xi w)
	-\lambda\|w-\widehat w_t^{(i)}\|^2
	\right\}.
	\end{align*}
	The calculation in Appendix \ref{app:one-step-affine-minimax} shows that $\mathcal{Q}_t(x,\cdot)$
	is quadratic with unique minimizer $u_t^\ast(x)=K_tx+k_t$ and Hessian
	$2\mathcal{R}_t$, where
	\[
	\mathcal{R}_t
	=
	R+B^\top \!P_{t+1}B
	+B^\top \!P_{t+1}\Xi
	(\lambda I-\Xi^\top \!P_{t+1}\Xi)^{-1}
	\Xi^\top \!P_{t+1}B.
	\]
	Hence
	\begin{align}\label{73}
	    \mathcal{Q}_t(x,u)-V_t(x)
	=
	(u-u_t^\ast(x))^\top
	\mathcal{R}_t
	(u-u_t^\ast(x)).
	\end{align}
	
	By Lemma \ref{lem:closed-tail-margin} and compactness of the Riccati tail,
	$\sup_{N,t}\|\mathcal{R}_t\|<\infty$. Thus, for $t<N-q$,
	\eqref{72}-\eqref{73} imply
	\begin{align}\label{74}
	    \mathcal{Q}_t(x,\pi_t^{(q)}(x))-V_t(x)
	\le
	C\rho_0^{\,2(N-t-1)}(1+\|x\|^2),
	\end{align}
	whereas the left-hand side is zero for $t\ge N-q$.
	
	It remains only to control the second moments appearing in \eqref{74}.
	Let $\widehat V_t^{(q)}$ denote the worst-case value associated with
	the fixed policy $\pi^{(q)}$. Since $\pi^{(q)}=\pi^\ast$ on the last
	$q$ stages,
	$\widehat V_{N-q}^{(q)}=V_{N-q}$.
	On $0\le t<N-q$, its quadratic and affine coefficients are generated
	by the stationary-policy Bellman map corresponding to
	$K_{\rm ss}x+k_{\rm ss}$. This map has the fixed point
	$(P_{\rm ss},s_{\rm ss})$, at which its linearization is block triangular with diagonal maps
	$\Delta\mapsto A_c^\top\Delta A_c$, $v\mapsto A_c^\top v$.
	Since $A_c$ is Schur, the same local contraction argument as in
	Lemma \ref{lem:riccati-tail-exp}, together with
	$P_{N-q}\to P_{\rm ss}$ and $s_{N-q}\to s_{\rm ss}$, yields, for all
	$q\ge q_0$,
	\[
	\sup_{0\le t\le N-q}
	\left(
	\|\widehat P_t^{(q)}-P_{\rm ss}\|
	+\|\widehat s_t^{(q)}-s_{\rm ss}\|
	\right)
	\le C\rho_0^q.
	\]
	Consequently the linear part of the corresponding worst-case closed loop is uniformly close to $A_c$. Using the Lyapunov matrix	employed in the proof of Lemma \ref{lem:uniform-product}, the uniform penalty bound of Lemma \ref{lem:closed-tail-margin}, and the fixed empirical covariance $\Sigma$, we therefore
	obtain
	\[
	\sup_{\substack{N\ge q\ge q_0\\0\le t\le N}}
	\mathbb{E}^{(q)}\|x_t\|^2
	\le
	C
	\left(1+\int_{\mathbb{R}^n}\|x\|^2\,d\mu_0(x)\right),
	\]
	where $\mathbb{E}^{(q)}$ denotes expectation under
	$\pi^{(q)}$ and its worst-case disturbance policy.
	
	Finally, apply the Bellman inequality defining $\mathcal{Q}_t$ along
	this worst-case closed loop and sum over $t$. The $V_t$ terms
	telescope, giving
	\[
	\begin{aligned}
		J_N^{\rm wc}(\pi^{(q)})-J_N^\ast
		&\le
		\sum_{t=0}^{N-q-1}
		\mathbb{E}^{(q)}
		\big[
		\mathcal{Q}_t(x_t,\pi_t^{(q)}(x_t))-V_t(x_t)
		\big] \\
		&\le
		C\left(1\!+\!\int_{\mathbb{R}^n}\!\!\|x\|^2\,\dd\mu_0(x)\!\right)\!
		\sum_{t=0}^{N-q-1}\!\!\rho_0^{\,2(N-t-1)} \\
		&\le
		\frac{C}{1-\rho_0^2}
		\left(1+\int_{\mathbb{R}^n}\|x\|^2\,\dd\mu_0(x)\right)
		\rho_0^{\,2q}.
	\end{aligned}
	\]
	The lower bound follows from the optimality of $\pi^\ast$.
	Taking $\rho_{\rm p}=\rho_0$ proves the claim.
\end{proof}

\section*{Acknowledgment}

The authors thank Enrique Zuazua for valuable discussions and helpful
suggestions.

\end{document}